\documentclass{amsart}

\usepackage{stmaryrd}
\usepackage{amssymb}
\usepackage{mathrsfs}
\usepackage[all]{xy}

\DeclareMathAlphabet{\mathpzc}{OT1}{pzc}{m}{it}

\xyoption{poly}
\entrymodifiers={+!!<0pt,\fontdimen22\textfont2>}
\def\cA{\mathscr{A}}
\def\cC{\mathscr{C}}
\def\cD{\mathscr{D}}
\def\cE{\mathscr{E}}
\def\cF{\mathscr{F}}
\def\cR{\mathscr{R}}
\def\cT{\mathscr{T}}
\def\cU{\mathscr{U}}
\def\cV{\mathscr{V}}

\def\BQ{\mathbb{Q}}
\def\BZ{\mathbb{Z}}

\def\fP{\mathfrak{P}}
\def\fQ{\mathfrak{Q}}
\def\fT{\mathfrak{T}}
\def\fU{\mathfrak{U}}
\def\fV{\mathfrak{V}}

\def\fc{\mathfrak{c}}
\def\fd{\mathfrak{d}}
\def\fp{\mathfrak{p}}
\def\ft{\mathfrak{t}}
\def\fu{\mathfrak{u}}

\def\add{\operatorname{add}}

\def\adots{\mathinner{\mkern1mu\raise1.0pt\vbox{\kern7.0pt\hbox{.}}\mkern2mu\raise4.0pt\hbox{.}\mkern2mu\raise7.0pt\hbox{.}\mkern1mu}}
\def\ast{{\textstyle *}}

\def\CC{\rho}

\def\coind{\operatorname{coind}}
\def\Coker{\operatorname{Coker}}

\def\dim{\operatorname{dim}}

\def\End{\operatorname{End}}

\def\Ext{\operatorname{Ext}}

\def\flength{\mathsf{fl}}
\def\Free{\operatorname{Free}}
\def\funct{G}

\def\Gr{\operatorname{Gr}}

\def\Hom{\operatorname{Hom}}

\def\ind{\operatorname{ind}}

\def\K{\operatorname{K}}

\def\mod{\mathsf{mod}}
\def\Mod{\mathsf{Mod}}

\def\objects{\operatorname{obj}}

\def\SL{\operatorname{SL}}

\def\split{\operatorname{split}}

\def\top{\operatorname{top}}

\newcommand{\cat}{\cC}
\newcommand{\shift}{\Sigma}
\newcommand{\cts}{\cT}
\newcommand{\indt}{T}
\newcommand{\obj}{t}
\newcommand{\modt}{\mod\cts}
\newcommand{\kosp}{K_0^{\split}(\cts)}
\newcommand{\ko}{K_0(\modt)}
\newcommand{\ctsu}{\cU}
\newcommand{\indu}{U}
\newcommand{\obju}{u}
\newcommand{\modu}{\mod\cts/\ctsu}
\newcommand{\kou}{K_0(\modu)}
\newcommand{\kospu}{K_0^{\split}(\cts/\ctsu)}
\newcommand{\catu}{\cat_{\indu}}
\newcommand{\shiftu}{\shift_{\indu}}
\newcommand{\indext}{\operatorname{ind}_{\cts}}
\newcommand{\coindext}{\operatorname{coind}_{\cts}}
\newcommand{\indexu}{\operatorname{ind}_{\cts/\ctsu}}
\newcommand{\coindexu}{\operatorname{coind}_{\cts/\ctsu}}
\newcommand{\xs}{x}
\newcommand{\gre}{\operatorname{Gr}_e}

\newcommand{\fl}{\rightarrow} \newcommand{\gfl}{\longrightarrow}
\newcommand{\dr}{\ar@{->}[r]}
\newcommand{\downeq}{\ar@{=}[d]}
\newcommand{\down}{\ar@{->}[d]}
\newcommand{\drm}{\ar@{^{(}->}[r]}
\newcommand{\ddrm}{\ar@{^{(}->}[rr]}
\newcommand{\downepi}{\ar@{->>}[d]}
\newcommand{\ddownepi}{\ar@{->>}[dd]}

\newtheoremstyle{bfupright head,slanted body}
                {2.5ex}{1ex}
                {\slshape}{0pt}{\bfseries}{.}{0.5em}
                {{\mdseries \thmnumber{(#2) }}\thmnote{#3}}

\newtheoremstyle{bfupright head,upright body}
                {2.5ex}{1ex}
                {\upshape}{0pt}{\bfseries}{.}{0.5em}
                {{\mdseries \thmnumber{(#2) }}\thmnote{#3}}

\newtheoremstyle{bfit head,upright body}
                {2.5ex}{1ex}
                {\upshape}{0pt}{\upshape}{.}{0.5em}
                {{\mdseries\thmnumber{(#2) }}
                {\bfseries\itshape\thmnote{\negthickspace#3}}}

\newtheoremstyle{it head,upright body}
                {2.5ex}{1ex}
                {\upshape}{0pt}{\upshape}{.}{0.5em}
                {{\mdseries\thmnumber{(#2) }}
                {\itshape\thmnote{\negthickspace#3}}}

\newtheoremstyle{fixed bf head,slanted body}
                {2.5ex}{1ex}{\slshape}
                {0pt}{\bfseries}{.}{0.5em}
                {{\mdseries \thmnumber{(#2) }}\thmname{#1}\thmnote{ (#3)}}

\newtheoremstyle{fixed bf head,upright body}
                {2.5ex}{1ex}{\upshape}
                {0pt}{\bfseries}{.}{0.5em}
                {{\mdseries \thmnumber{(#2) }}\thmname{#1}\thmnote{ (#3)}}

\newtheoremstyle{indented paragraph}
                {2.5ex}{1ex}
                {\upshape}{\parindent}{\upshape}{}{0pt}
                {\thmnote{#3 }}

\theoremstyle{fixed bf head,slanted body}
\newtheorem{Lemma}{Lemma}[section]
\newtheorem{Proposition}[Lemma]{Proposition}
\newtheorem{Theorem}[Lemma]{Theorem}
\newtheorem{Corollary}[Lemma]{Corollary}

\newtheorem*{Lemma*}{Lemma}
\newtheorem*{Proposition*}{Proposition}
\newtheorem*{Theorem*}{Theorem}
\newtheorem*{Corollary*}{Corollary}

\theoremstyle{fixed bf head,upright body}
\newtheorem{Definition}[Lemma]{Definition}
\newtheorem{Example}[Lemma]{Example}
\newtheorem{Remark}[Lemma]{Remark}

\newtheorem{Setup}[Lemma]{Setup}

\newtheorem{Notation}[Lemma]{Notation}

\newtheorem*{Definition*}{Definition}
\newtheorem*{Example*}{Example}
\newtheorem*{Remark*}{Remark}
\newtheorem*{Observation*}{Observation}
\newtheorem*{Setup*}{Setup}
\newtheorem*{Conjecture*}{Conjecture}
\newtheorem*{Construction*}{Construction}
\newtheorem*{Question*}{Question}
\newtheorem*{Notation*}{Notation}

\theoremstyle{bfupright head,slanted body}

\newtheorem*{res*}{}

\theoremstyle{bfit head,upright body}
\newtheorem*{com*}{}

\theoremstyle{bfupright head,upright body}
\newtheorem{bfhpg}[Lemma]{}               
\newtheorem*{bfhpg*}{}

\theoremstyle{it head,upright body}
               
\newtheorem*{ithpg*}{}

\theoremstyle{indented paragraph}

\begin{document}

\setlength{\parindent}{0pt}
\setlength{\parskip}{7pt}

\title[Infinite Caldero-Chapoton]{A Caldero-Chapoton map for
infinite clusters}

\author{Peter J\o rgensen}
\address{School of Mathematics and Statistics,
Newcastle University, Newcastle upon Tyne NE1 7RU,
United Kingdom}
\email{peter.jorgensen@ncl.ac.uk}
\urladdr{http://www.staff.ncl.ac.uk/peter.jorgensen}

\author{Yann Palu}
\address{Department of Pure Mathematics,
University of Leeds, Leeds LS2 9JT, United Kingdom}
\email{ypalu@maths.leeds.ac.uk}
\urladdr{http://www.amsta.leeds.ac.uk/~ypalu}

\keywords{Calabi-Yau reduction, cluster category, cluster map, cluster
structure, cluster tilting subcategory, coindex, Dynkin type
$A_{\infty}$, exchange pair, exchange triangle, Fomin-Zelevinsky
mutation, Grassmannian, index, K-theory, $\SL_2$-tiling}

\subjclass[2010]{13F60, 16G10, 16G20, 16G70, 18E30}

\begin{abstract} 
  We construct a Caldero--Chapoton map on a triangulated category with
  a cluster tilting subcategory which may have infinitely many
  indecomposable objects.

  The map is not necessarily defined on all objects of the
  triangulated category, but we show that it is a (weak) cluster map
  in the sense of Buan-Iyama-Reiten-Scott.  As a corollary, it induces
  a surjection from the set of exceptional objects which can be
  reached from the cluster tilting subcategory to the set of cluster
  variables of an associated cluster algebra.

  Along the way, we study the interaction between Calabi--Yau
  reduction, cluster structures, and the Caldero--Chapoton map.

  We apply our results to the cluster category $\cD$ of Dynkin type
  $A_{\infty}$ which has a rich supply of cluster tilting
  subcategories with infinitely many indecomposable objects.  We show
  an example of a cluster map which cannot be extended to all of
  $\cD$.

  The case of $\cD$ also permits us to illuminate results by
  Assem--Reutenauer--Smith on $\SL_2$-tilings of the plane.
\end{abstract}

\maketitle

\setcounter{section}{-1}
\section{Introduction}
\label{sec:introduction}

The Caldero--Chapoton map was introduced in \cite{CC}. It formalises the
connection between the cluster algebras of Fomin--Zelevinsky \cite{FZ}
and the cluster categories of Buan--Marsh--Reineke--Reiten--Todorov
\cite{BMRRT}.

Cluster algebras are commutative rings with connections to areas as
diverse as Calabi-Yau algebras, integrable systems, Poisson geometry,
and quiver representations, see \cite{Z}.  Cluster categories are,
loosely speaking, certain categories of re\-pre\-sen\-ta\-ti\-ons of
finite dimensional algebras which were introduced to ``categorify''
cluster algebras (see the surveys~\cite{KellerSurvey1},
\cite{KellerSurvey2} and the references given there).
Indeed, using the Caldero--Chapoton map, Caldero
and Keller~\cite{CK2} established a bijection between the
indecomposable rigid objects of a cluster ca\-te\-go\-ry and the
cluster variables of the corresponding cluster algebra (see also the
appendix of \cite{BCKMRT}).

The Caldero--Chapoton map was generalized in~\cite{Palu} to
$2$-Calabi--Yau triangulated categories with cluster tilting objects.
However, it is also interesting to consider ca\-te\-go\-ri\-es which
have no cluster tilting objects but only cluster tilting
subcategories.  In this situation, the cluster tilting subcategories
have infinitely many in\-de\-com\-po\-sa\-ble objects, so one cannot
take the direct sum and thereby obtain cluster tilting objects; see
for instance \cite{BIRS} and \cite{KellerReiten2}.  A relevant example
of this is the cluster category $\cD$ of Dynkin type $A_\infty$,
studied in~\cite{HJ}, which has no cluster tilting objects, but does
have a rich supply of cluster tilting subcategories.

In this paper, we show that the methods of~\cite{Palu} apply to
$2$-Calabi--Yau triangulated categories with cluster tilting
subcategories.  Let $k$ be an algebraically closed field, let $\cat$
be a $k$-linear Hom-finite Krull--Schmidt triangulated category which
is $2$-Calabi--Yau, and let $\cts$ be a cluster tilting subcategory
\cite{KellerReiten2}.

In Section~\ref{sec:CC}, we define a Caldero--Chapoton map associated
with $\cts$.  In this context, the Caldero--Chapoton map may not be
defined on all the objects of the category $\cat$. It is therefore not
a cluster character in the sense of \cite{Palu}.

Nonetheless, in Section~\ref{sec:Cluster_maps} we prove that it is a
(weak) cluster map in the sense of~\cite{BIRS} (Theorem
\ref{thm:cluster_map}).  As a corollary, it induces a surjection from
the set of reachable exceptional objects in $\cat$ to the set of
cluster variables of an associated cluster algebra.

In Section~\ref{sec:HJ}, for the particular case of the category
$\cD$, this yields an explicit example of a cluster map in the sense
of~\cite{BIRS} which cannot be nicely extended to a cluster character
(Theorem~\ref{thm:nogo}).

Sections~\ref{sec:IY} through \ref{sec:locally bounded} are devoted
to Calabi--Yau reductions. After recalling the definition and some
results from~\cite{IyamaYoshino}, we describe the connection between
modules over
cluster tilting subcategories of $\cat$ and its Calabi-Yau reductions.
We also study the interactions between Calabi--Yau reduction, cluster
structures, and the Caldero--Chapoton map.  There are two purposes to
this: First, the results are used for the category $\cD$ to show
positivity of its cluster maps; this is what makes some of them
impossible to extend to cluster characters.  Secondly, the results are
used in Section~\ref{sec:locally bounded} to prove an independent
result: If the cluster tilting subcategory $\cts$ is locally bounded,
then the formulae for the Caldero--Chapoton map and for the cluster
character of a well-chosen Calabi--Yau reduction coincide.

The final Section \ref{sec:tilings} has an application of our results
to the study of $\SL_2$-tilings as introduced by Assem, Reutenauer,
and Smith in \cite{ARS}.

\medskip
{\em Acknowledgement. }  We thank Bernhard Keller for suggesting that
the Caldero-Chapoton map can be made to work in the context of cluster
tilting subcategories with infinitely many indecomposable objects.

This enabled us to improve the results in a previous version of the
paper which used Calabi-Yau reduction to go down to subquotient
categories with cluster tilting objects, and ``bootstrapping'' to lift
their cluster characters to the original category.  We are grateful to
Osamu Iyama, Ralf Schiffler, Dong Yang, and Andrei Zelevinsky for
comments to the previous version.

We thank Idun Reiten for answering a question on \cite{BIRS}.

The second-named author would like to thank Robert Marsh
and the EPSRC for his current post-doctoral position.
He would also like to thank Raf Bocklandt for an invitation
to Newcastle during which part of this work was carried out.

\section{The Caldero-Chapoton formula}
\label{sec:CC}

In this section, we define a version of the Caldero-Chapoton map for cluster
tilting subcategories with infinitely many indecomposable objects.
Before doing so, we will review the necessary ingredients to make sure
that everything works as needed in this situation.

\begin{Setup}
\label{set:blanket}
Let $k$ be an algebraically closed field.

In Sections \ref{sec:CC} through \ref{sec:locally bounded}, the
following notation is used:

$\cat$ is a $k$-linear $\Hom$-finite Krull--Schmidt triangulated
category which is $2$-Calabi-Yau.  Its shift functor is $\shift$.

$\cts$ is a cluster tilting subcategory of $\cat$ and $\indt = \ind
\cts$ is its set of (isomorphism classes of) indecomposable objects.
Note that $\cts = \add \indt$.
\end{Setup}

\begin{bfhpg}
[Modules over $\cts$]
\label{bfhpg:funct}
Following~\cite{KellerReiten2}, the category $\Mod\,\cts$ of
$\cts$-modules is the category of $k$-linear contravariant functors
from $\cts$ to the category of $k$-vector spaces.  A $\cts$-module $M$
is said to be finitely presented if it admits a presentation by
representable functors:
\[
  \cts(-,t_1) \rightarrow \cts(-,t_0) \rightarrow M \rightarrow 0.
\]
In particular, finitely presented modules are functors from $\cts$ to
the category of finite dimensional $k$-vector spaces.  The full
subcategory of finitely presented modules is denoted by $\mod\,\cts$,
and is an abelian category (see~\cite[2.1]{KellerReiten2}).

By \cite[thm.\ 2.2]{BMR}, \cite[2.1, prop.]{KellerReiten2},
\cite[cor.\ 4.4]{KoenigZhu} the functor
\begin{center}
\begin{tabular}{ccc}
  $\cat$ & $\stackrel{\funct}{\longrightarrow}$ & $\mod\,\cts$, \\[3mm]
  $c$ & $\longmapsto$ & $\cat(-,\Sigma c)\big|_{\cts}$
\end{tabular}
\end{center}
induces an equivalence of categories $\cat/[\cts]
\simeq \mod\,\cts$.  The square brackets denote the ideal of morphisms
which factor through an object in $\cts$.
\end{bfhpg}

\begin{bfhpg}
[The results of \cite{Palu}]
\label{bfhpg:Palu}
A number of results were established in \cite{Palu} under the
conditions of Setup \ref{set:blanket}, with the additional assumption
that the set $\indt$ was finite.  This assumption is unnecessary for many
of the results.  To wit, lemmas 3.1, 3.3, 4.2 and proposition 4.3 of
\cite{Palu} certainly remain true under Setup \ref{set:blanket} alone,
with the same proofs.

Moreover, indices and coindices of objects in $\cat$ were introduced in
\cite{Palu} and further studied in~\cite{DK}.  If $c$ is an object of $\cat$, then pick distinguished
triangles
\[
  \obj_1 \rightarrow \obj_0 \rightarrow c, \;\;\;\;
  c \rightarrow \Sigma^2 \obj^0 \rightarrow \Sigma^2 \obj^1
\]
with $\obj_i, \obj^i \in \cts$; cf.\ \cite[2.1 prop]{KellerReiten2} or \cite[lem.\ 3.2(1)]{KoenigZhu}.
Define the index and the coindex relative to $\cts$ as the classes
\[
  \ind_{\cts} c = [\obj_0] - [\obj_1], \;\;\;\;
  \coind_{\cts} c = [\obj^0] - [\obj^1]
\]
in the split Grothendieck group $\K_0^{\split}(\cts)$.  Here the
square brackets indicate $\K$-classes.  The subscripts of $\ind_\cts$
and $\coind_\cts$ will be omitted when there is no danger of
confusion.  This definition differs slightly from
\cite{Palu} where index and coindex took values in the usual
Grothendieck group $\K_0(\mod\,\cts)$.  However, \cite{Palu}, lemma
2.1 (parts 1, 2, 4) and proposition 2.2 remain true under Setup
\ref{set:blanket} alone, again with the same proofs.
\end{bfhpg}

\begin{bfhpg}
[Finite length objects]
The simple objects in $\Mod\,\cts$ are precisely the objects $S_\obj = \top \cts(-,\obj)$ for
$\obj \in \indt$, see \cite[prop.\ 2.3(b)]{AusRepII}.

For $\obj \in \indt$, there exists a right almost split morphism $b
\rightarrow \obj$ in $\cts$, and hence $S_\obj \in \mod\,\cts$ by
\cite[cor.\ 2.6]{AusRepII}.  Let us show that $b \rightarrow \obj$
exists.

Set $\indt' = \indt \setminus \obj$ and write $\cA = \add \indt'$.  By
\cite[thm.\ 5.3]{IyamaYoshino}, there exists a unique indecomposable
$\obj^\ast \neq \obj$ in $\cat$ for which $\cts^\ast = \add(\indt'
\cup \obj^\ast)$ is a cluster tilting subcategory.  Moreover,
$(\cts^\ast , \cts)$ is a so-called $\cA$-mutation pair by \cite[thm.\
5.3]{IyamaYoshino} again.  One consequence is that there exists a
distinguished triangle $x \rightarrow a \rightarrow \obj \rightarrow
\Sigma x$ in $\cat$ with $x \in \cts^\ast$ and $a \in \cA$; see
\cite[def.\ 2.5]{IyamaYoshino}.  If $s \neq \obj$ is an indecomposable
object in $\indt$, then $s \in \cts^\ast$ and so $\cat(s,\Sigma x) =
0$.  It follows that any morphism $s \rightarrow \obj$ factors through
$a \rightarrow \obj$.  Now pick $n$ morphisms in $\cat(\obj,\obj)$
which generate the radical of $\cat(\obj,\obj)$ as a
right-$\cat(\obj,\obj)$-module.  They define a morphism $\obj^n
\rightarrow \obj$, and it is not hard to see that $a \oplus \obj^n
\rightarrow \obj$ is right almost split.

So each $S_\obj$ is in $\mod\,\cts$.  Conversely, let $S$ be a simple
object of $\mod\,\cts$.  The category $\mod\,\cts$ has enough
projectives, so there is an epimorphism from a projective object,
$\cts(-,\obj) \rightarrow S$, where $\obj \in \cts$; we can clearly assume
that $\obj$ is indecomposable.  This implies that $S_\obj$ is a quotient of
$S$ in $\Mod\,\cts$, but both of these objects are simple objects of
$\mod\,\cts$ so it follows that $S \cong S_\obj$.

Hence $\Mod\,\cts$ and $\mod\,\cts$ have the same simple objects, namely
the $S_\obj$ for $\obj \in \indt$.

Using that the inclusion functor $\mod\,\cts \rightarrow \Mod\,\cts$ is
exact, see \cite[sec.\ III.2]{AusRepDim}, and that $\mod\,\cts$ is
closed under extensions in $\Mod\,\cts$, it follows that $\Mod\,\cts$
and $\mod\,\cts$ have the same finite length objects.  We denote the
category of these by $\flength\,\cts$.
\end{bfhpg}

\begin{bfhpg}
[K-theory]
\label{bfhpg:K}
\begin{enumerate}

  \item  It is clear that there is an isomorphism
\[
  \K_0^{\split}(\cts) \cong \Free(\indt).
\]

\smallskip

  \item  There is the following homomorphism,
\begin{center}
\begin{tabular}{rcl}
  $\K_0(\mod\,\cts)$ & $\stackrel{\theta_{\cts}}{\longrightarrow}$ & $\K_0^{\split}(\cts)$, \\[3mm]
  $[\funct c]$ & $\longmapsto$ & $\coind_{\cts} \Sigma c - \ind_{\cts} \Sigma c$,
\end{tabular}
\end{center}
where the subscript of $\theta_{\cts}$ will be omitted when
there is no danger of confusion.  Namely, the expression $\coind
\Sigma c - \ind \Sigma c$ depends only on $\funct c$ by \cite[lem.\
2.1(4)]{Palu}, and respects the relations of $\K_0(\mod\,\cts)$ by
\cite[lem.\ 3.1 and prop.\ 2.2]{Palu}.

The embedding $\flength\,\cts \rightarrow \mod\,\cts$ induces a
homomorphism $\K_0(\flength\,\cts) \rightarrow \K_0(\mod\,\cts)$ which
we can compose with $\theta$; by abuse of notation, the composition
will also be denoted $\theta$.

We will use this homomorphism in place of the bilinear form $\langle
-,- \rangle_a$ of \cite{Palu}. If the quiver of $\cts$ is locally
finite and has no loops, then they are related via the following
formula:
\[
 \theta(-) = \sum_{\obj \in \indt} \langle S_\obj ,- \rangle_a [\obj].
\]

Note that if $\cts$ belongs to a cluster structure on $\cat$ in the
sense of \cite[sec.\ II.1]{BIRS}, then each $\obj \in \indt$ sits in
so-called exchange triangles $\obj^\ast \rightarrow c \rightarrow
\obj$, $\obj \rightarrow c' \rightarrow \obj^\ast$ where $\obj^\ast$
is the unique replacement of $\obj$ which leaves $\cts$ a cluster
tilting subcategory, and where $c, c' \in \add(\indt \setminus
\obj)$. In that case we have $S_\obj = G \obj^\ast$ and a computation
shows
\[
  \theta([S_\obj]) = [c] - [c'].
\]

\smallskip

  \item We introduce variables $x_\obj$ for $\obj \in \indt$.  If
  $\alpha = \sum_\obj \alpha_\obj \, [\obj]$ is an element of
  $\K_0^{\split}(\cts)$, then we write
\[
  x^{\alpha} = \prod_\obj x_\obj^{\alpha_\obj};
\]
this is a well defined monomial which we view as an element of the
rational function field $\BQ(x_\obj)_{\obj \in \indt}$ generated by the
$x_\obj$.

\end{enumerate}
\end{bfhpg}

\begin{bfhpg}
[Grassmannians]
\label{bfhpg:Grassmannians}
For each $\obj \in \indt$ the functor $S_\obj \in \flength\,\cts$ is
zero at each indecomposable different from $\obj$.  If $M \in
\flength\,\cts$ satisfies $[M] = e$ in $\K_0(\flength\,\cts)$ where
$e = \sum_\obj e_\obj \, [S_\obj]$, then $M$ has a composition series
with quotients $S_\obj$ for which $e_\obj \neq 0$.  It follows that
the functor $M$ is zero at each indecomposable $\obj$ with $e_\obj =
0$.  So $M$ is supported at the finitely many $\obj \in \indt$ with
$e_\obj \neq 0$.

Now let $N \in \mod\,\cts$ and consider subfunctors $M \subseteq N$
where $M$ has finite length and satisfies $[M] = e$ in
$\K_0(\flength\,\cts)$.  Then $M$ is supported at the finitely many
$\obj \in \indt$ with $e_\obj \neq 0$, so $M$ can be viewed as an
$m$-dimensional sub-vector space of the finite dimensional $k$-vector
space $E = \bigoplus_{ \{ \obj | e_\obj \neq 0 \} } N(\obj)$, and $m$ is
determined by $e$.  There is a Grassmannian of $m$-dimensional
subspaces of $E$, and it follows that there is a Grassmannian
$\Gr_e(N)$ whose points parametrize the subfunctors $M
\subseteq N$ where $M$ has finite length and satisfies $[M] = e$.
Both Grassmannians are algebraic spaces.

It is also possible to introduce auxiliary algebraic spaces in the
manner of \cite[sec.\ 5.1]{Palu}.  Namely, let $\ell, m \in \cat$ be
such that $\funct \ell$, $\funct m$ have finite length, and suppose
that there are distinguished triangles
\begin{equation}
\label{equ:triangles}
  m \stackrel{i}{\rightarrow} b \stackrel{p}{\rightarrow} \ell, \;\;\;\;
  \ell \stackrel{i'}{\rightarrow} b' \stackrel{p'}{\rightarrow} m
\end{equation}
in $\cat$.  Note that $\funct b$, $\funct b'$ also have finite length.
Let $e, f \in \K_0(\flength\,\cts)$ and set
\begin{align*}
  X_{e,f} & = \{\, E \subseteq \funct b
                  \,|\, [(\funct i)^{-1}E] = e,\: [(\funct p)E] = f \,\}, \\
  Y_{e,f} & = \{\, E' \subseteq \funct b'
                  \,|\, [(\funct i')^{-1}E'] = f,\: [(\funct p')E'] = e \,\}.
\end{align*}
There are morphisms of algebraic spaces
\begin{center}
\begin{tabular}{rcrclcl}
  $X_{e,f}$ & $\rightarrow$ & $\Gr_e(\funct m)$ & $\times$ & $\Gr_f(\funct \ell)$ & $\leftarrow$ & $Y_{e,f}$, \\[3mm]
  $E$ & $\mapsto$ & $\big( (\funct i)^{-1}E$ & , & $(\funct p)E \big)$, & & \\[3mm]
  & & $\big( (\funct p')E'$ & , & $(\funct i')^{-1}E' \big)$ & $\mapsfrom$ & $E'$
\end{tabular}
\end{center}
for which each fibre is an affine space $k^n$.  This property of the
fibres is due to \cite{CC}, see their lemma 3.8, and it can be proved
in the present situation by following the proof of \cite[lem.\
18]{Palu2} to obtain a free transitive action of the algebraic group
$k^n$ on a fibre.  In fact, for a point $E \in X_{e,f}$ one sets $V =
(Gp)E$ and $W = \Coker((\funct i)^{-1}E \hookrightarrow \funct m)$,
and it is $(\mod\,\cts)(V,W)$, viewed as an algebraic group isomorphic
to $k^n$, which acts on the fibre containing $E$.  Note that $V$ and
$W$ have finite length in $\mod\,\cts$, so $(\mod\,\cts)(V,W)$ is
indeed finite dimensional over $k$.

There are now two cases of interest.

{\em Case 1: The triangles \eqref{equ:triangles} are split.  } Then
there is a (split) exact sequence $0 \rightarrow \funct m
\stackrel{\funct i}{\rightarrow} \funct b \stackrel{\funct
p}{\rightarrow} \funct \ell \rightarrow 0$ which implies
\[
  \Gr_g(\funct b) = \coprod_{e+f=g} X_{e,f}.
\]
Moreover, the morphism $X_{e,f} \rightarrow \Gr_e(\funct m) \times
\Gr_f(\funct \ell)$ is surjective since $M \oplus L$ is sent to
$(M,L)$.  The fibres are affine spaces, so we get
\[
  \chi(\Gr_e(\funct m) \times \Gr_f(\funct \ell)) = \chi(X_{e,f})
\]
where $\chi$ is (ordinary or $\ell$-adic) Euler characteristic.

The last two equations give the first and third equality in the
following key formula.
\begin{align}
\label{equ:CC}
\nonumber
  \sum_{e+f=g} \chi(\Gr_e(\funct m) \times \Gr_f(\funct \ell))
  & = \sum_{e+f=g} \chi(X_{e,f}) \\
  & = \chi \Big( \coprod_{e+f=g} X_{e,f} \Big)
    = \chi(\Gr_g(\funct b))
\end{align}
Note that the sums are finite since the relevant algebraic spaces are
only non-empty for $e$, $f$, $g$ in the positive cone of
$\K_0(\flength\,\cts)$.

{\em  Case 2:  We have $\Ext_{\cat}^1(\ell,m) = 1$ and the triangles
\eqref{equ:triangles} are not split.  }  The morphisms from $X_{e,f}$
and $Y_{e,f}$ induce a morphism from the disjoint union
\[
  X_{e,f} \: {\textstyle \coprod} \: Y_{e,f}
  \rightarrow \Gr_e(\funct m) \times \Gr_f(\funct \ell).
\]
It follows from \cite[prop.\ 4.3]{Palu} that this morphism is
surjective and that each point in the target is the image either of a
point from $X_{e,f}$ or a point from $Y_{e,f}$, but not both.  Hence
the fibres are still affine spaces $k^n$, so
\[
  \chi(\Gr_e(\funct m) \times \Gr_f(\funct \ell)) 
  = \chi(X_{e,f} \: {\textstyle \coprod} \: Y_{e,f})
  = \chi(X_{e,f}) + \chi(Y_{e,f}).
\]
Following \cite[sec.\ 5.1]{Palu} in setting
\[
  X_{e,f}^g = X_{e,f} \cap \Gr_g(\funct b), \;\;\;\;
  Y_{e,f}^g = Y_{e,f} \cap \Gr_g(\funct b')
\]
for $g \in \K_0(\flength\,\cts)$ hence gives another key formula,
\begin{equation}
\label{equ:Palu}
  \chi(\Gr_e(\funct m) \times \Gr_f(\funct \ell))
  = \sum_g \chi(X_{e,f}^g) + \chi(Y_{e,f}^g).
\end{equation}

It is clear that
\begin{equation}
\label{equ:chi}
  \chi(\Gr_g(\funct b)) = \sum_{e,f} \chi(X_{e,f}^g), \;\;\;\;
  \chi(\Gr_g(\funct b')) = \sum_{e,f} \chi(Y_{e,f}^g).
\end{equation}
The sums in equations \eqref{equ:Palu} and \eqref{equ:chi}
are also finite.  Namely, if $e$, $f$ are fixed, then each
of $X^g_{e,f}$ and $Y^g_{e,f}$ is only non-empty for finitely many
values of $g$.  And if $g$ is fixed, then each of $X^g_{e,f}$ and
$Y^g_{e,f}$ is only non-empty for finitely many values of $(e,f)$.
\end{bfhpg}

The following is a version of \cite[lem.\ 5.1]{Palu}, and a version of
the proof of that result works.  We stay within the situation of
Paragraph \ref{bfhpg:Grassmannians}.

\begin{Lemma}
\label{lem:Palu}
If $e, f, g \in \K_0(\flength\,\cts)$ are such that $X_{e,f}^g \neq
\emptyset$, then
\[
  \theta(g) - \coind \Sigma b
  = \theta(e+f) - \coind \Sigma m - \coind \Sigma \ell.
\]
\end{Lemma}

\begin{bfhpg}
[The Caldero-Chapoton formula for cluster tilting subcategories]
\label{bfhpg:CC}
Let $c \in \cat$.  Using the material from the previous paragraphs, we
give the following version of the Caldero-Chapoton formula from
\cite[sec.\ 3]{CC}.
\[
  \CC^{\cts}(c) = x^{-\coind_{\cts} \Sigma c} \sum_e \chi(\Gr_e(\funct c)) x^{\theta_{\cts}(e)}
\]
The sum is over $e$ in $\K_0(\flength\,\cts)$, and the superscript of
$\CC^\cts$ will be omitted when there is no danger of confusion.

For an arbitrary $c$, it is not clear that $\CC(c)$ is well-defined.
However, if $\funct c$ has finite length in $\mod\,\cts$ then $\CC(c)$
is a well-defined element of $\BZ[x_\obj, x_\obj^{-1}]_{\obj
\in \indt}$, and hence of $\BQ(x_\obj)_{\obj \in \indt}$, because the
sum is finite.  Namely, since the functor $\funct c$ has finite
length, it is only supported at finitely many $\obj \in \indt$.  By
the deliberations of Paragraph \ref{bfhpg:Grassmannians}, this implies
that there are only finitely many classes $e$ in
$\K_0(\flength\,\cts)$ for which $\funct c$ has a submodule $M$ with
$[M] = e$.  And hence only finitely many classes $e$ in
$\K_0(\flength\,\cts)$ for which $\Gr_e(\funct c) \neq \emptyset$.
\end{bfhpg}

The following version of \cite[cor.\ 3.7]{CC} works under Setup
\ref{set:blanket}, without the assumption that $\indt$ is finite.

\begin{Proposition}
\label{pro:CC}
Suppose that $m, \ell \in \cat$ have $\funct m$, $\funct \ell$ of
finite length in $\mod\,\cts$.  Then $\CC(m)$, $\CC(\ell)$ are
elements of $\BZ[x_\obj, x_\obj^{-1}]_{\obj \in \indt}$ which satisfy
\[
  \CC(m \oplus \ell) = \CC(m)\CC(\ell).
\]
\end{Proposition}

\begin{proof}
We have $\CC(m), \CC(\ell) \in \BZ[x_\obj, x_\obj^{-1}]_{\obj \in
\indt}$ by Paragraph \ref{bfhpg:CC}, and the formula is an immediate
consequence of equation \eqref{equ:CC}.
\end{proof}

The following version of \cite[thm.\ 1.4]{Palu} also works under Setup
\ref{set:blanket} alone.

\begin{Proposition}
\label{pro:Palu}
Suppose that $m, \ell \in \cat$ have $\dim \Ext^1_{\cat}(m,\ell) = 1$
and that $\funct m$, $\funct \ell$ have finite length in $\mod\,\cts$.

Let
\[
  m \rightarrow b \rightarrow \ell, \;\;\;\;
  \ell \rightarrow b' \rightarrow m
\]
be non-split distinguished triangles in $\cat$.  Then $\CC(m)$,
$\CC(\ell)$, $\CC(b)$, $\CC(b')$ are elements of $\BZ[x_\obj,
x_\obj^{-1}]_{\obj \in \indt}$ which satisfy
\[
  \CC(m)\CC(\ell) = \CC(b) + \CC(b').
\]
\end{Proposition}

\begin{proof}
The objects $b, b'$ also have $\funct b$, $\funct b'$ of finite
length, so $\CC(m), \CC(\ell), \CC(b), \CC(b') \in \BZ[x_\obj,
x_\obj^{-1}]_{\obj \in \indt}$ by Paragraph
\ref{bfhpg:CC}.

To get the equation, we use the proof of \cite[thm.\ 1.4]{Palu} given
in \cite[sec.\ 5.1]{Palu} which also works in the present situation.
Only a modest computation is necessary; we show it here.
\begin{align*}
  \CC(m)\CC(\ell)
  & = x^{-\coind \Sigma m}\sum_e \chi(\Gr_e \funct m)x^{\theta e}
      \times x^{-\coind \Sigma \ell}\sum_f \chi(\Gr_f \funct \ell)x^{\theta f} \\
  & = x^{-\coind \Sigma m - \coind \Sigma \ell}
      \sum_{e,f} \chi(\Gr_e \funct m)\chi(\Gr_f \funct \ell)x^{\theta(e+f)} \\
  & = x^{-\coind \Sigma m - \coind \Sigma \ell}
      \sum_{e,f} \chi \big( \Gr_e(\funct m) \times \Gr_f(\funct \ell) \big) x^{\theta(e+f)} \\
  & \stackrel{\rm (a)}{=} x^{-\coind \Sigma m - \coind \Sigma \ell}
      \sum_{e,f,g} \big( \chi(X_{e,f}^g) + \chi(Y_{e,f}^g) \big) x^{\theta(e+f)} \\
  & \stackrel{\rm (b)}{=}
      \sum_{e,f,g} \chi(X_{e,f}^g)x^{\theta(g) - \coind \Sigma b}
      + \sum_{e,f,g} \chi(Y_{e,f}^g)x^{\theta(g) - \coind \Sigma b'} \\
  & \stackrel{\rm (c)}{=}
      x^{-\coind \Sigma b} \sum_g \chi(\Gr_g \funct b)x^{\theta g}
      + x^{-\coind \Sigma b'} \sum_g \chi(\Gr_g \funct b')x^{\theta g} \\
  & = \CC(b) + \CC(b')
\end{align*}
Here (a) is by Equation \eqref{equ:Palu}, (b) is by Lemma
\ref{lem:Palu} and its analogue for $Y$ and $b'$, and (c) is by
Equation \eqref{equ:chi}. 
\end{proof}

\section{Cluster maps}
\label{sec:Cluster_maps}

Recall that we work under Setup \ref{set:blanket}.  In this section,
we prove that the Caldero-Chapoton map associated with
a cluster tilting subcategory is a (weak) cluster map in the
sense of~\cite{BIRS}.

\begin{Setup}
\label{set:Cluster_maps}
In this section we assume that the cluster tilting subcategories which
can be reached by finitely many mutations from $\cts$ form a cluster
structure in the sense of \cite[sec.\ II.1]{BIRS}.

Note that this forces the quiver $Q_\cts$ of $\cts$ to be locally finite. 

The set of indecomposable objects belonging to cluster tilting
subcategories which can be reached from $\cts$ will be denoted by $E$,
and we will write $\cE = \add E$.
\end{Setup}

\begin{bfhpg}
[Cluster maps]
\label{bfhpg:cluster_maps}
We recall the following from \cite[sec.\ IV.1]{BIRS}.
A map
\[
  \varphi : \objects \cE \rightarrow \BQ(x_\obj)_{\obj \in \indt}
\]
is called a cluster map if it satisfies the following conditions.
\begin{enumerate}

  \item  $\varphi$ is constant on isomorphism classes.

\smallskip

  \item If $c_1, c_2 \in \cE$ then $\varphi(c_1 \oplus c_2) =
    \varphi(c_1)\varphi(c_2)$.

\smallskip

  \item If $m, \ell \in \objects \cE$ are indecomposable objects with
    $\dim_k \Ext^1_{\cat}(m,\ell) = 1$ and $b, b' \in \objects \cE$ are
    such that there are non-split distinguished triangles
\[
  m \rightarrow b \rightarrow \ell, \;\;\;\;
  \ell \rightarrow b' \rightarrow m
\]
in $\cat$, then $\varphi(m)\varphi(\ell) = \varphi(b) +
\varphi(b')$.

\smallskip

  \item There is a cluster tilting subcategory $\cts'$ which can be
  reached from $\cts$ for which $\{\, \varphi(\obj') \,|\, \obj' \in
  \ind \cts' \,\}$ is a transcendence basis of the field
  $\BQ(x_\obj)_{\obj \in \indt}$.

\end{enumerate}
\end{bfhpg}

\begin{Theorem}
\label{thm:cluster_map}
The Caldero-Chapoton formula for cluster tilting subcategories
gives a well-defined cluster map
\[
  \CC : \objects \cE \rightarrow \BQ(x_\obj)_{\obj \in \indt}
\]
satisfying $\CC(\obj) = x_\obj$ for each $\obj \in \indt$. 
\end{Theorem}

\begin{proof}
Let us first show that $\funct(\varepsilon)$ has finite length for
$\varepsilon \in \objects\cE$.  It is enough to show that $\funct(e)$
has finite length when $e$ is an indecomposable object belonging to a
cluster tilting subcategory $\cts'$ which can be obtained from $\cts$
by successive mutations at finitely many objects $\obj_1, \ldots,
\obj_m$.  But $\ind \cts'$ is equal to $\indt \setminus \{ \obj_1,
\ldots, \obj_m \}$ united with finitely many indecomposable objects,
and it follows that $\cat(\obj,\Sigma e) = 0$ for $\obj \in \indt
\setminus \{ \obj_1, \ldots, \obj_m \}$.  Hence $\funct(e)$ is
supported on $\{ \obj_1, \ldots, \obj_m \}$ and so it has finite
length.

Let us now prove the theorem.  The equality $\CC(\obj) = x_\obj$ for
$\obj \in \indt$ is shown by direct computation.  Paragraph
\ref{bfhpg:CC} combined with finite length of $\funct(\varepsilon)$
shows that $\CC$ is well-defined on $\objects \cE$ and takes values in
$\BZ[x_\obj, x_\obj^{-1}]_{\obj \in
\indt}$ and hence in $\BQ(x_\obj)_{\obj \in \indt}$.  Property (i)
of cluster maps is clear, property (ii) holds by Proposition
\ref{pro:CC}, and property (iii) by Proposition \ref{pro:Palu}.

Property (iv) holds with $\cts' = \cts$ since $\CC(\obj) = x_\obj$.
\end{proof}

\begin{bfhpg}
[Cluster algebras]
Recall that cluster algebras were originally introduced by Fomin and
Zelevinsky in \cite{FZ}.  We will not recall their definition, but
merely remind the reader that a cluster algebra is generated by
cluster variables which form overlapping sets
of fixed cardinality called clusters.  

The clusters in \cite{FZ} were finite, but cluster algebras with
countable clusters are defined in \cite[sec.\ IV.1]{BIRS}.
\end{bfhpg}

\begin{Corollary}
Suppose that $\cts$ has countably many indecomposable objects and let
$\cA$ be the cluster algebra associated with the quiver $Q_\cts$.

The map $\CC$ induces a surjection from $E$ to the cluster variables
of $\cA$.  Under the surjection, the indecomposable objects of a
cluster tilting subcategory which can be reached from $\cts$ are sent
to a cluster.
\end{Corollary}

\begin{proof}
Immediate from Theorem \ref{thm:cluster_map} and \cite[thm.\
IV.1.2(a)]{BIRS}. 
\end{proof}

\section{Calabi--Yau reductions}
\label{sec:IY}

The technique of Calabi--Yau reductions, introduced by Iyama and
Yoshino in \cite[sec.\ 4]{IyamaYoshino}, yields fully faithful
functors between the module categories of different cluster tilting
subcategories.  In this section, we recall the construction of
Calabi-Yau reductions and give a nice description of these functors.

\begin{Setup}
In this section we let $\indu \subseteq \indt$ be a subset for which
$\ctsu = \add \indu$ is functorially finite in $\cat$.
\end{Setup}

For a subcategory $\cV$ of $\cat$, we let $^\perp \cV$ (resp.\
$\cV^\perp$) denote the full subcategory of $\cat$ whose objects are
left (resp.\ right) Hom-orthogonal to each object of $\cV$.

\begin{bfhpg}
[Calabi-Yau reductions]
\label{bfhpg:IY}
We recall the following from \cite[sec.\ 4]{IyamaYoshino}.  Set
$$
\catu =\,^\perp(\shift \ctsu)/[\ctsu]
$$
and note that $\cat_\emptyset = \cat$ and $\cat_\indt = 0$.  The
category $\catu$ is triangulated. Its suspension functor $\shiftu$ is
defined on objects as follows: Let $c \in {}^\perp(\shift \ctsu)$ and
let $c \fl \obju_c$ be a minimal left-$\ctsu$-approximation. Complete it to a
distinguished triangle in $\cat$
\begin{equation}
\label{equ:Yann_Delta}
c \fl \obju_c \fl \shiftu c \stackrel{\sigma_c}{\fl} \shift c,
\end{equation}
where $\shiftu c$ belongs to $^\perp(\shift \ctsu)$ by the Wakamatsu
lemma.
Let $c \fl d \fl e \fl \shift c$ be a distinguished triangle in $\cat$, with
$c,d,e \in$ $^\perp(\shift \ctsu)$. Since $e$ belongs to
$^\perp(\shift \ctsu)$, the composition
$e \fl \shift c \fl \shift \obju_c$ vanishes and the identity of $c$
induces a morphism of triangles
$$
\xymatrix{
c \dr \downeq & d \dr \down & e \dr \down & \shift c \downeq \\
c \dr & \obju_c \dr & \shiftu c \dr & \shift c \lefteqn{.}
}
$$
The image in $\catu$ of the complex $c \fl d \fl e \fl \shiftu c$ is a
distinguished triangle and all distinguished triangles in $\catu$ are
isomorphic to such.

For $c,d \in$ $^\perp(\shift \ctsu)$, applying the functor $\cat(d,-)$
to the triangle \eqref{equ:Yann_Delta} shows that composition by
$\sigma_c$ induces an isomorphism
\begin{equation}
\label{equ:Yann_ast}
\catu(d , \shiftu c) \stackrel{\simeq}{\gfl} \cat(d , \shift c).
\end{equation}
Moreover, 
$\cat_{\indu}$ is a $k$-linear $\Hom$-finite Krull-Schmidt
2-Calabi--Yau category.
The projection functor ${}^{\perp}(\Sigma \ctsu)
\stackrel{\pi}{\fl} \catu$ induces a bijection between the cluster
tilting subcategories of $\cat$ containing $\ctsu$ and the cluster
tilting subcategories of $\catu$, and the bijection is compatible with
mutation at indecomposable objects outside $\ctsu$.
\end{bfhpg}

\begin{Example}
For each $n$, the cluster category of type $A_n$ is a Calabi--Yau
reduction of the cluster category $\cD$ of type $A_{\infty}$. See
Lemma~\ref{lem:Ng} for more details on this situation.
\end{Example}

\begin{bfhpg}
[Modules over cluster tilting subcategories]
\label{bfhpg:cluster_tilting}
The following formula contains an equivalence of categories which
comes from the third isomorphism theorem and a fully faithful functor
which is obvious.
\[
  \frac{{}^{\perp}(\Sigma \ctsu) / [\ctsu]}{[\cts / [\ctsu]]}
  \simeq \frac{{}^{\perp}(\Sigma \ctsu)}{[\cts]}
  \hookrightarrow \frac{\cat}{[\cts]}
\]
Abbreviating $\cts / [\ctsu]$ to $\cts / \ctsu$, this gives a fully
faithful functor $\catu / [\cts / \ctsu] \hookrightarrow \cat /
[\cts]$.  By way of the functor $\funct$ of Paragraph \ref{bfhpg:funct}
and its analogue $\funct_{\indu}$ for the cluster tilting subcategory
$\cts / \ctsu$ of $\catu$, this translates into a fully faithful
functor
$$
\modu\! \stackrel{\hspace{6pt}\pi^\ast}{\hookrightarrow} \modt
$$
which sends $\catu(-,\shiftu c)|_{\cts / \ctsu}$ to $\cat(-,\shift
c)|_{\cts}$. 

The situation is depicted in the following diagram.
$$
\xymatrix{
\cts \downepi_{\pi|_{\cts}} \drm        &
^\perp(\shift \ctsu) \downepi^\pi \drm  &
\cat \ddownepi                          \\
\cts/\ctsu \drm                         &
\catu \downepi                          &
                                        \\
                                        &
\catu/[\cts/\ctsu] \drm \down_\simeq    &
\cat/[\cts] \down^\simeq                \\
                                        &
\modu \drm_-{\pi^\ast}      &
\modt
}
$$
The isomorphism \eqref{equ:Yann_ast} of Paragraph \ref{bfhpg:IY} shows
that the functor $\pi^\ast$ is indeed e\-qui\-va\-lent to the one
induced by composition with $\pi$: For $c \in {}^\perp(\shift \ctsu)$
the $\cts$-modules $\catu(-,\shiftu c)|_{\cts / \ctsu} \circ
\pi|_{\cts}$ and $\cat(-,\shift c)|_{\cts}$ are isomorphic.
\end{bfhpg}

\begin{Proposition}
\label{pro:pi}
\begin{enumerate}

  \item  The fully faithful functor $\pi^{\ast}$ is exact.

\smallskip

  \item  The essential image of $\pi^{\ast}$ is the $\cts$-modules
    vanishing on $\ctsu$, so $\pi^{\ast}$ identifies finitely
    presented $\cts/\ctsu$-modules with finitely presented
    $\cts$-modules vanishing on $\ctsu$.

\smallskip

  \item  Given $M \in \mod \cts / \ctsu$, the functor $\pi^{\ast}$
    induces a bijection between isomorphism classes of submodules of
    $M$ in $\mod \cts / \ctsu$ and isomorphism classes of submodules
    of $\pi^{\ast}M$ in $\mod \cts$.

\end{enumerate}
\end{Proposition}

\begin{proof}
(i) An exact sequence in $\mod \cts/\ctsu$ has the form
\begin{equation}
\label{equ:exact_sequence}
  \catu(-,\shiftu c)|_{\cts / \ctsu}
  \rightarrow \catu(-,\shiftu d)|_{\cts / \ctsu}
  \rightarrow \catu(-,\shiftu e)|_{\cts / \ctsu}
\end{equation}
for a distinguished triangle $c \rightarrow d \rightarrow e$ in
$\catu$ (see~\cite[lem. 3.1]{Palu}) which may be taken to come from a distinguished triangle in
$\cat$ with entries in ${}^{\perp}(\Sigma \ctsu)$.  But $\pi^{\ast}$
sends \eqref{equ:exact_sequence} to $\cat(-,\shift c)|_{\cts}
\rightarrow \cat(-,\shift d)|_{\cts} \rightarrow \cat(-,\shift
e)|_{\cts}$ which is also exact.

(ii) Let $c \in \cat$.  By the $2$-Calabi--Yau property, $\cat(-,\shift c)|_{\cts}$
vanishes on $\ctsu$ if and only if $c$ belongs to $^\perp(\shift
\ctsu)$.

(iii)  Since $\pi^\ast M$ is a $\cts$-module vanishing on $\ctsu$,
all its sub-$\cts$-modules also vanish on $\ctsu$.
\end{proof}

\section{Calabi--Yau reductions and the Caldero-Chapoton map}
\label{sec:CYclmaps}

This section shows an aspect of the interaction between Calabi-Yau
reductions and the Caldero-Chapoton map.

\begin{Lemma}\label{lem: functorially finite}
Let $R \subseteq \indt$ be such that $\indt \setminus R$ is finite.
Then $\cR = \add R$ is functorially finite in $\cat$.
\end{Lemma}

\begin{Remark*}
The proof we will give of the lemma works under weaker assumptions
than the standing Setup \ref{set:blanket}.  Namely, it is only needed
that $\cat$ has a Serre functor (not that it is $2$-Calabi--Yau) and
that $\cts$ is $n$-cluster tilting for some $n$ (not that it is
cluster tilting).
\end{Remark*}

\begin{proof}
We only prove that $\cR$ is contravariantly finite.  Covariant
finiteness then follows by duality.  Since $\cts$ is contravariantly
finite in $\cat$, it is enough to prove that each $t \in \indt
\setminus R$ has a right-$\cR$-approximation. By propositions 2.10 and 2.11
of~\cite{IyamaYoshino}, any object in $\cts$ has a sink map. Let
$b_0\oplus c_0 \stackrel{a_0}{\gfl} t$ be a sink map in $\cts$ with
$b_0 \in \add(\indt \setminus R)$ and $c_0 \in \add R$.  If $b_0$ is
zero, then $a_0$ is the approximation we are looking for.  Otherwise, let
$b_1\oplus c_1 \stackrel{a_1}{\gfl} b_0$ be a sink map in $\cts$, with
the same conventions as above.  We then inductively construct sink
maps $b_j\oplus c_j \stackrel{a_j}{\gfl} b_{j-1}$, with the convention
that $a_j$ is zero if $b_{j-1}$ is zero.  Since $\cat$ is Hom-finite,
there is some $m$ for which the induced composition $b_m \fl \cdots
\fl b_1 \fl b_0 \fl t$ vanishes. One easily checks that
$\bigoplus_{j=0}^m c_j \fl t$ is a right-$\cR$-approximation.
\end{proof}

\begin{Lemma}\label{lem: reduction cluster str}
Assume that the cluster tilting
subcategories which can be reached from $\cts$ form a cluster
structure in $\cat$, in the sense of \cite[sec.\ II.1]{BIRS}.
Then, for any subset $\indu \subseteq \indt$
such that $\cU = \add \indu$ is functorially finite,
the cluster tilting subcategories which can be reached from
$\cts/\ctsu$ form a cluster structure in the Calabi--Yau reduction
$\catu$.
\end{Lemma}

\begin{proof}
 By~\cite[thm. II.1.6]{BIRS}, it is enough to check that
the cluster tilting subcategories reachable from $\cts/\ctsu$ in
$\catu$ have no loops or 2-cycles. Moreover, since the Calabi--Yau
reduction is compatible with mutation, it is enough to check
the following: If $\cts$ has no loops or 2-cycles in $\cat$, then
$\cts/\ctsu$ has no loops or 2-cycles in $\catu$. This follows easily
from the fact that the exchange triangles in $\catu$ are the images
of the exchange triangles in $\cat$.
\end{proof}

\begin{Setup}
\label{set:CYred}
In the rest of this section,
we let $\indu \subseteq \indt$ be a subset for which
$\ctsu = \add \indu$ is functorially finite in $\cat$,
and we assume that the cluster tilting subcategories
which can be reached from $\cts/\ctsu$ form a cluster structure
in $\catu$.

Note that hence, the results of Sections \ref{sec:CC} and
\ref{sec:IY} apply in $\cat$ and the results of
Section \ref{sec:Cluster_maps} apply in $\catu$.
\end{Setup}

\begin{Notation}
The set of indecomposable objects belonging to cluster tilting
subcategories of $\catu$ which can be reached from $\cts / \ctsu$ will
be denoted by $F$, and we will write $\cF = \add F$.
\end{Notation}

\begin{Remark}\label{rk: identification}
 Objects in $\catu$ are identified with the objects
of ${}^{\perp}(\Sigma
\ctsu) \subseteq \cat$
without direct summands in $\ctsu$.
Under this identification, the objects in $\cat$
corresponding to objects in $\cF$
are sent to finite length modules by the functor $\funct$;
indeed, they are sent to finite length $\cts/\ctsu$-modules
by the proof of Theorem \ref{thm:cluster_map}.
The following definition
therefore makes sense:
\end{Remark}

\begin{Definition}
Let the map
\[
  \overline{\CC}^{\cts} : \objects\,\cF
  \rightarrow \BQ(x_\obj)|_{\obj \in \indt \setminus \indu}
\]
be given by $\overline{\CC}^{\cts}(f) = \CC^{\cts}(f)|_{x_\obju=1
\:\mbox{\tiny for}\: \obju \in \indu}$. 
\end{Definition}

\begin{Proposition}
\label{pro:reduction}
The map $\overline{\CC}^{\cts} : \objects\,\cF \rightarrow
\BQ(x_\obj)|_{\obj \in \indt \setminus \indu}$ is a cluster map.
\end{Proposition}

\begin{proof}
We must prove that the map $\overline{\CC}^{\cts}$
satisfies the conditions of Paragraph
\ref{bfhpg:cluster_maps}.  Conditions (i) and (ii)
follow immediately from the corresponding properties of $\CC^{\cts}$.

To see (iii), let $m, \ell \in \cF$ be indecomposable objects with
$\dim_k \Ext_{\catu}^1(m,\ell) = 1$ and suppose that there are $b, b'
\in \cF$ and non-split distinguished triangles
\[
  m \rightarrow b \rightarrow \ell, \;\;\;\;
  \ell \rightarrow b' \rightarrow m
\]
in $\catu$.  By the isomorphism \eqref{equ:Yann_ast}
of Paragraph \ref{bfhpg:IY}
we get $\dim_k
\Ext_{\cat}^1(m,\ell) = 1$ and so the two triangles are isomorphic in
$\catu$ to the images of distinguished triangles in $\cat$:
\[
  m \rightarrow b \oplus \obju \rightarrow \ell, \;\;\;\;
  \ell \rightarrow b' \oplus \obju' \rightarrow m
\]
where $\obju, \obju' \in \ctsu$. By Remark~\ref{rk: identification}
we have that $Gm$, $G\ell$ have finite length, so Propositions
\ref{pro:CC} and \ref{pro:Palu} give
\[
  \CC^{\cts}(m) \CC^{\cts}(\ell) 
  = \CC^{\cts}(b \oplus \obju) + \CC^{\cts}(b' \oplus \obju')
  = \CC^{\cts}(b) \CC^{\cts}(\obju) + \CC^{\cts}(b') \CC^{\cts}(\obju')
\]
whence
$$
  \overline{\CC}^{\cts}(m) \overline{\CC}^{\cts}(\ell)
  = \overline{\CC}^{\cts}(b) + \overline{\CC}^{\cts}(b').
$$

Finally, (iv) follows as in the proof of Theorem \ref{thm:cluster_map}
since $\overline{\CC}^{\cts}(\obj)= x_\obj$ for each $\obj\in \indt
\setminus \indu$.
\end{proof}

\begin{Theorem}
\label{thm:CC_reduction}
The maps $\CC^{\cts/\ctsu}$ and $\overline{\CC}^{\cts}$ coincide on
$\objects\,\cF$.
\end{Theorem}

\begin{proof}
Both are cluster maps by Theorem \ref{thm:cluster_map} and Proposition
\ref{pro:reduction}.  Therefore it is enough to check that they
coincide on the indecomposable objects of $\cts / \ctsu$, cf.\
\cite[Lemma 5.3]{Palu}.  But if $\obj \in \indt\backslash\indu$, then
$\CC^{\cts/\ctsu}$ and $\overline{\CC}^{\cts}$ both take the value
$x_\obj$ on $\obj$.
\end{proof}

\begin{Corollary}\label{Corollary: pointwise finite}
Assume that the cluster tilting subcategories which can be reached
from $\cts$ form a cluster structure in $\cat$.
Then the map $\CC^{\cts}$ is ``pointwise finite'' in the following sense:
For each $\ell \in \objects \cE$, there exists a cluster tilting object
$R$ in a Calabi--Yau reduction of $\cat$ such that $\CC^{\cts}(\ell) =
X^R_\ell$, where $X^R$ is the cluster character of~\cite{Palu}.
\end{Corollary}

\begin{proof}
We may take $\ell \in \objects \cE$ to be indecomposable, and it sits
in some cluster tilting subcategory obtained from $\cts$ by a finite
number of mutations.  Therefore the spaces $\cat(t, \shift\ell)$
vanish for all but finitely many $t_1,\ldots,t_r \in T$. There exist
$t_{r+1},\ldots,t_m \in T$ such that $\CC^\cts(\ell)$ belongs to
$\BQ(x_{t_1},\ldots, x_{t_m}) \subset
\BQ(x_t)|_{t\in T}$. Let $R$ be the direct sum
$t_1\oplus\cdots\oplus t_m$ and let 
$\indu$ be $\indt \setminus \{\, t_1, \ldots, t_m \,\}$.
Lemma~\ref{lem: functorially finite} ensures
that the subcategory $\ctsu = \add \indu$
is functorially finite.
By Lemma~\ref{lem: reduction cluster str},
the cluster tilting object $R$ defines a cluster
structure on the Calabi--Yau reduction $\catu$.
We can now apply Theorem~\ref{thm:CC_reduction}
to get the desired result.
\end{proof}

\begin{Remark}
 P-G. Plamondon recently defined, in~\cite{Plamondon},
a Caldero--Chapoton map associated with any rigid object
(in a not necessarily Hom-finite triangulated category)
and proved that it behaves like a cluster map.
Assume that the rigid object is a direct summand of a cluster tilting
object. It would be interesting to adapt the methods
used above in order to relate the Caldero--Chapoton map
of~\cite{Plamondon} with the one on the Calabi--Yau
reduction in which the rigid object becomes cluster tilting.
\end{Remark}

\section{Calabi-Yau reductions in the locally bounded case}
\label{sec:locally bounded}

In the previous section, Theorem \ref{thm:CC_reduction} showed a
degree of compatibility between Calabi-Yau reductions and the
Caldero-Chapoton map.  This section goes further under the additional
assumption that the cluster tilting subcategory is locally bounded.
Following~\cite{LenzingReiten}, a full subcategory $\cV$ of $\cat$ is
called locally bounded if, for each indecomposable object $v$ in
$\cV$, there are only finitely many (isomorphism classes of)
indecomposable objects $w$ in $\cV$ such that $\cat(v,w) \neq 0$ or
$\cat(w,v)\neq 0$.

\begin{Setup}
\label{set:first}
In this section we let $\indu \subseteq \indt$ be a subset for which
$\ctsu = \add \indu$ is functorially finite in $\cat$.  (See Lemma
\ref{lem: functorially finite} for examples of such subsets.)
\end{Setup}

Let
\[
  \kappa : \kospu \hookrightarrow \kosp
\]
be the canonical inclusion.

\begin{Lemma}
\label{lem:index}
Let $c \in {}^{\perp}(\Sigma \ctsu)$; in particular, $c$ can be viewed
as an object of $\catu$ as well as of $\cat$.
\begin{enumerate}
  \item If $c \in {}^\perp\ctsu$, then
$\indext \shift c = \kappa(\indexu \shiftu c)$.

\smallskip

  \item If $c \in {}^\perp(\shift^2\ctsu)$, then
$\coindext \shift c = \kappa(\coindexu \shiftu c)$.
\end{enumerate}
\end{Lemma}

\begin{proof}
(i) In the distinguished triangle $c \fl \obju_c \fl \shiftu c \fl
\shift c$ we have $\obju_c = 0$ since $c \in \,^\perp\ctsu$, so
$\shiftu c$ and $\shift c$ are isomorphic in $\catu$.  Pick a
distinguished triangle $\obj_1 \fl \obj_0 \fl \shift c \fl \shift
\obj_1$ with $\obj_i \in \cT$ and $\obj_0 \rightarrow \Sigma c$
right-minimal; note that $\Sigma c \rightarrow \Sigma \obj_1$ is then
left-minimal.  Since $c \in {}^{\perp}(\Sigma \ctsu)$ it follows that
$\Sigma c \in \ctsu^{\perp}$ because $\cat$ is $2$-Calabi-Yau.  And $c
\in {}^{\perp}\ctsu$ so $\Sigma c \in {}^{\perp}(\Sigma \ctsu)$.
Therefore $\obj_0$ and $\obj_1$ have no direct summands in $\ctsu$.

But $\obj_0,\obj_1,\shift c$ belong to $^\perp(\shift \ctsu)$, so
there is an induced distinguished triangle $\obj_1 \fl \obj_0 \fl
\shift c \fl \shiftu \obj_1$ in $\catu$ which, by the beginning of the
proof, is $\obj_1 \fl \obj_0 \fl \shiftu c \fl \shiftu \obj_1$.  Since
the objects $\obj_i$ have no direct summands in $\ctsu$, the morphism
$\kappa$ sends $[\obj_0]-[\obj_1]$ in $\kospu$ to $[\obj_0]-[\obj_1]$
in $\kosp$.

(ii) This is dual to (i) because the coindex of $\shift c$ in
$(\cat,\cts)$ is the index of $\shift^{-1} c$ in
$(\cat^\text{op},\cts^\text{op})$, $\shift^{-1}$ is the shift functor
in $\cat^\text{op}$, and $^\perp(\shift^2 \ctsu)=\ctsu^\perp$ by the
$2$-Calabi-Yau property.
\end{proof}

\begin{Setup}
In the rest of this section, we assume that the cluster tilting
sub\-ca\-te\-go\-ry $\cts$ is locally bounded.
\end{Setup}

\begin{Remark}
\label{rmk:locally_bounded}
Let $c \in \cat$.  There is a right-$\cts$-approximation $\obj_0 \fl c$ and a
left-$\cts$-approximation $c \fl \obj^0$. Since $\cts$ is locally bounded,
this implies that there are only finitely many indecomposable objects
$\obj\in\indt$ such that $\cat(\obj,c) \neq 0$ or $\cat(c,\obj) \neq
0$.

This implies that the functor $\funct c = \cat(-,\Sigma c)|_{\cts}$
has finite length in $\mod\,\cts$.  By Paragraph \ref{bfhpg:CC} the
Caldero--Chapoton map $\CC^{\cts}$ is defined on all the objects of
$\cat$ whence $\CC^{\cts}$ is a cluster character~\cite{Palu}.  It
also means that any subset $\indu \subseteq \indt$ can be used in
Setup \ref{set:first} because $\ctsu = \add\,\indu$ is always
functorially finite in $\cat$.

Now let $s_1,\ldots,s_n\in\cat$, without non-zero direct summands in
$\cts$, be such that the modules $\cat(-,\shift s_i)|_\cts$ are the
simple composition factors of $\cat(-,\shift c)|_\cts$.  Consider the
following condition on $\indu\subseteq\indt$:
\begin{equation}
\label{equ:U_conditions}
  \mbox{
    $c \in \,^\perp(\shift \ctsu) \cap {}^\perp(\shift^2\ctsu)$
    and
    $s_i \in \,^\perp\ctsu \cap \,^\perp(\shift\ctsu) \cap
    \,^\perp(\shift^2\ctsu)$
    for each $i$;}
\end{equation}
there are such sets with $\indt \setminus \indu$ finite
by the first part of the Remark, and as noted,
$\ctsu = \add\,\indu$ is necessarily functorially finite.

We will draw some conclusions about this situation in Lemma
\ref{lem:exponents} and Theorem \ref{thm:reductionCCmap}.  Observe
that since $c \in {}^{\perp}(\Sigma \ctsu)$, we can view $c$ 
as an object of both $\cat$ and $\catu$.
\end{Remark}

Consider the (non-commutative) square
$$
\xymatrix{
\kou \ddrm^{K_0(\pi^\ast)} \down_{\theta_{\cts/\ctsu}} & &
\ko \down^{\theta_\cts}                                \\
\kospu \ddrm_\kappa                                    & &
\kosp
}
$$
where $\theta_\cts$, $\theta_{\cts/\ctsu}$ are defined in Paragraph
\ref{bfhpg:K} and $\kappa$ above.

\begin{Lemma}
\label{lem:exponents}
Under the above conditions \eqref{equ:U_conditions}, let $e$ be a
class in $\kou$ coming from a submodule of $\catu(-,\shiftu
c)|_{\cts/\ctsu}$. Then we have $$\theta_\cts\circ K_0(\pi^\ast)(e) =
\kappa\circ\theta_{\cts/\ctsu}(e).  $$
\end{Lemma}

\begin{proof}
Since the objects $c$, $s_1,\ldots,s_n$ belong to the subcategory
$^\perp(\shift\ctsu)$, the functors $\catu(-,\shiftu
c)|_{\cts/\ctsu}$, $\catu(-,\shiftu s_i)|_{\cts/\ctsu}$ are well
defined and finitely presented. Moreover, by
Paragraph~\ref{bfhpg:cluster_tilting}, their images under $\pi^\ast$
are $\cat(-,\shift c)|_\cts$, $\cat(-,\shift s_i)|_\cts$.  It follows
from Proposition \ref{pro:pi}(iii) that the modules $\catu(-,\shiftu
s_i)|_{\cts/\ctsu}$ are the simple composition factors of
$\catu(-,\shiftu c)|_{\cts/\ctsu}$ in $\modu$.  Denote by $e_i$ the
class of $\catu(-,\shiftu s_i)|_{\cts/\ctsu}$ in $\kou$. So the class
$e$ is a linear combination of the classes $e_i$ and it is enough to
prove the proposition for the $e_i$.  However,
$$K_0(\pi^\ast)(e_i)
= \left[\cat(-,\shift s_i)|_\cts\right],$$
so
$$\theta_\cts\circ K_0(\pi^\ast)(e_i) =
\coindext \shift s_i - \indext \shift s_i.$$

On the other hand,
$$\theta_{\cts/\ctsu}(e_i) =
\coindexu \shiftu s_i - \indexu \shiftu s_i.$$
The result thus follows from Lemma~\ref{lem:index} because
$s_i\in\,^\perp\ctsu\cap
\,^\perp(\shift\ctsu)\cap\,^\perp(\shift^2\ctsu)$.
\end{proof}

\begin{Theorem}
\label{thm:reductionCCmap}
Under the above conditions \eqref{equ:U_conditions}, we have
\[
  \CC^\cts(c) = \CC^{\cts/\ctsu}(c).
\]
\end{Theorem}

\begin{proof}
By Lemma~\ref{lem:index}(ii) we have $\coindext\shift c =
\kappa(\coindexu\shiftu c)$, so we can concentrate on the sums in the
Caldero-Chapoton formula.

We know from Paragraph~\ref{bfhpg:cluster_tilting} that the functor
$\pi^\ast$ sends the finitely presented $\cts/\ctsu$-module
$\catu(-,\shiftu c)|_{\cts / \ctsu}$ to the $\cts$-module
$\cat(-,\shift c)|_{\cts}$, and from Proposition \ref{pro:pi}(iii)
that it induces a bijection from the isomorphism classes of submodules
of $\catu(-,\shiftu c)|_{\cts / \ctsu}$ in $\modu$ to the isomorphism
classes of submodules of $\cat(-,\shift c)|_{\cts}$ in $\modt$.

Let $e$ be in $\kou$.  Then $\pi^\ast$ induces an isomorphism of
varieties from $\gre \big( \catu(-,\shiftu c)|_{\cts/\ctsu} \big)$ to
$\operatorname{Gr}_{K_0(\pi^\ast)(e)} \big( \cat(-,\shift c)|_\cts \big)$.
Moreover, the classes in the Grothendieck group
$\ko$ corresponding to non-vanishing terms in
$\CC^\cts(c)$ are all of the form $K_0(\pi^\ast)(e)$, for some
$e\in\kou$.  Therefore it only remains to be checked that
$$
  \xs^{\theta_{\cts/\ctsu}(e)} =
  \xs^{\theta_\cts\big(K_0(\pi^\ast)e\big)}
$$
for elements $e \in \kou$ representing
submodules of $\catu(-,\shiftu c)|_{\cts/\ctsu}$.
But this holds by Lemma~\ref{lem:exponents}.
\end{proof}

\section{The cluster category of Dynkin type $A_{\infty}$}
\label{sec:HJ}

This section applies the results of the previous sections to $\cD$, the
cluster category of type $A_{\infty}$ of \cite{HJ}.  It has no
cluster tilting objects, but does have a rich supply of cluster
tilting subcategories as explained below.

We obtain cluster maps on $\cD$, establish some properties, most
importantly po\-si\-ti\-vi\-ty, and show that some cluster maps cannot
be nicely extended to cluster characters.

\begin{bfhpg}
[The cluster category of type $A_{\infty}$]
\label{bfhpg:A_infty}
This category was studied in \cite{HJ}; we denote it by $\cD$.  We
still let $\cts$ denote a cluster tilting subcategory and write $\indt
= \ind \cts$.  Let us recall some features of this situation.

The results of Sections \ref{sec:CC} and \ref{sec:Cluster_maps} apply
to $\cD$ and $\cts$ because Setups \ref{set:blanket} and
\ref{set:Cluster_maps} hold by \cite[rmk.\ 1.2 and thm.\ 5.2]{HJ}.  In
particular, Theorem \ref{thm:cluster_map} gives a cluster map
\[
  \CC : \objects \cE \rightarrow \BQ(x_\obj)|_{\obj \in \indt}
\]
with $\CC(\obj) = x_\obj$ for $\obj \in \indt$.  Recall that $\cE =
\add E$ where $E$ is the set of in\-de\-com\-po\-sa\-ble objects of
$\cD$ which can be reached by finitely many mutations from $\cts$.

The AR quiver of $\cD$ is $\BZ A_{\infty}$ by \cite[rmk.\ 1.4]{HJ}.
The following coordinate system on the quiver is useful.
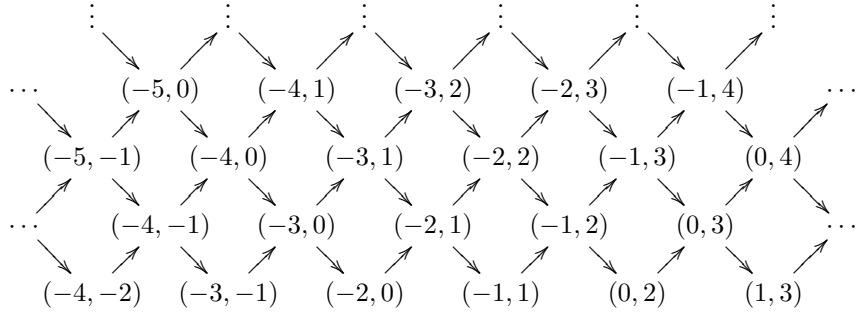
\[
  \xymatrix @-3.5pc @! {
    & \vdots \ar[dr] & & \vdots \ar[dr] & & \vdots \ar[dr] & & \vdots \ar[dr] & & \vdots \ar[dr] & & \vdots & \\
    \cdots \ar[dr]& & (-5,0) \ar[ur] \ar[dr] & & (-4,1) \ar[ur] \ar[dr] & & (-3,2) \ar[ur] \ar[dr] & & (-2,3) \ar[ur] \ar[dr] & & (-1,4) \ar[ur] \ar[dr] & & \cdots \\
    & (-5,-1) \ar[ur] \ar[dr] & & (-4,0) \ar[ur] \ar[dr] & & (-3,1) \ar[ur] \ar[dr] & & (-2,2) \ar[ur] \ar[dr] & & (-1,3) \ar[ur] \ar[dr] & & (0,4) \ar[ur] \ar[dr] & \\
    \cdots \ar[ur]\ar[dr]& & (-4,-1) \ar[ur] \ar[dr] & & (-3,0) \ar[ur] \ar[dr] & & (-2,1) \ar[ur] \ar[dr] & & (-1,2) \ar[ur] \ar[dr] & & (0,3) \ar[ur] \ar[dr] & & \cdots\\
    & (-4,-2) \ar[ur] & & (-3,-1) \ar[ur] & & (-2,0) \ar[ur] & & (-1,1) \ar[ur] & & (0,2) \ar[ur] & & (1,3) \ar[ur] & \\
               }
\]
This establishes a bijection between indecomposable objects and arcs
$(m,n)$ connecting non-neighbouring integers.  Hence the set $\indt$ of
indecomposable objects corresponds to a collection $\fT$ of arcs.

It was shown in \cite[thms.\ 4.3 and 4.4]{HJ} that $\cts = \add \indt$
being a cluster tilting subcategory is equivalent to $\fT$ being a
maximal collection of non-crossing arcs which is locally finite or has
a fountain.  Locally finite means that for each integer $n$, only
finitely many arcs of the form $(m,n)$ and $(n,p)$ are in $\fT$.
Having a fountain means that there is an integer $n$ such that $\fT$
contains infinitely many arcs of the form $(m,n)$ and infinitely many
of the form $(n,p)$.  The next figures show examples of these two
types of $\fT$.
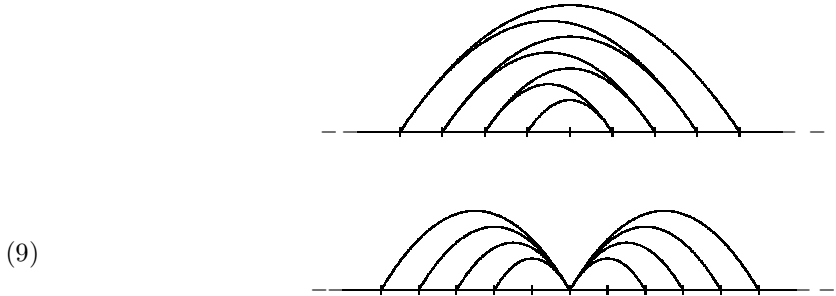
\[
\vcenter{
  \xymatrix @-4.75pc @! {
       \rule{0ex}{10ex} \ar@{--}[r]
     & *{}\ar@{-}[r]
     & *{\rule{0.1ex}{0.8ex}} \ar@{-}[r] \ar@/^3.5pc/@{-}[rrrrrrr]\ar@/^4.0pc/@{-}[rrrrrrrr]
     & *{\rule{0.1ex}{0.8ex}} \ar@{-}[r] \ar@/^2.5pc/@{-}[rrrrr]\ar@/^3.0pc/@{-}[rrrrrr]
     & *{\rule{0.1ex}{0.8ex}} \ar@{-}[r] \ar@/^1.5pc/@{-}[rrr]\ar@/^2.0pc/@{-}[rrrr]
     & *{\rule{0.1ex}{0.8ex}} \ar@{-}[r] \ar@/^1.0pc/@{-}[rr]
     & *{\rule{0.1ex}{0.8ex}} \ar@{-}[r] 
     & *{\rule{0.1ex}{0.8ex}} \ar@{-}[r]
     & *{\rule{0.1ex}{0.8ex}} \ar@{-}[r]
     & *{\rule{0.1ex}{0.8ex}} \ar@{-}[r]
     & *{\rule{0.1ex}{0.8ex}} \ar@{-}[r]
     & *{}\ar@{--}[r]
     & *{}
                    }
}
\]
{}
\begin{equation}
\label{equ:fountain}
\vcenter{
  \xymatrix @-4.0pc @! {
       \rule{0ex}{7.5ex} \ar@{--}[r]
     & *{}\ar@{-}[r]
     & *{\rule{0.1ex}{0.8ex}} \ar@{-}[r] \ar@/^2.5pc/@{-}[rrrrr]
     & *{\rule{0.1ex}{0.8ex}} \ar@{-}[r] \ar@/^2.0pc/@{-}[rrrr]
     & *{\rule{0.1ex}{0.8ex}} \ar@{-}[r] \ar@/^1.5pc/@{-}[rrr]
     & *{\rule{0.1ex}{0.8ex}} \ar@{-}[r] \ar@/^1.0pc/@{-}[rr]
     & *{\rule{0.1ex}{0.8ex}} \ar@{-}[r] 
     & *{\rule{0.1ex}{0.8ex}} \ar@{-}[r] 
     & *{\rule{0.1ex}{0.8ex}} \ar@{-}[r]
     & *{\rule{0.1ex}{0.8ex}} \ar@{-}[r] \ar@/^-1.0pc/@{-}[ll]
     & *{\rule{0.1ex}{0.8ex}} \ar@{-}[r] \ar@/^-1.5pc/@{-}[lll]
     & *{\rule{0.1ex}{0.8ex}} \ar@{-}[r] \ar@/^-2.0pc/@{-}[llll]
     & *{\rule{0.1ex}{0.8ex}} \ar@{-}[r] \ar@/^-2.5pc/@{-}[lllll]
     & *{}\ar@{--}[r]
     & *{}
                    }
}
\end{equation}
In the case of the fountain, $\cts$ contains infinitely many
indecomposable objects on each of the diagonal halflines $(-,n)$ and
$(n,-)$ in the AR quiver of $\cD$.

We will say that the arc $\ft = (s,t)$ spans the arc $\fu = (u,v)$ if
$s \leq u < v < t$ or $s < u < v \leq t$.  Note that an arc does not
span itself.
\end{bfhpg}

\begin{Theorem}
\label{thm:rho_domain}
Consider the cluster map $\CC : \objects \cE \rightarrow
\BQ(x_\obj)|_{\obj \in \indt}$.  The subcategory $\cE$ is determined
as follows. 
\begin{enumerate}

  \item  If $\fT$ is locally finite, then $\cE = \cD$.

\smallskip

  \item  If $\fT$ has a fountain at $n$, then $\cE$ is $\add$ of the
  indecomposable objects which are on or below one of the halflines
  $(-,n)$ and $(n,-)$ in the AR quiver of $\cD$.

\smallskip

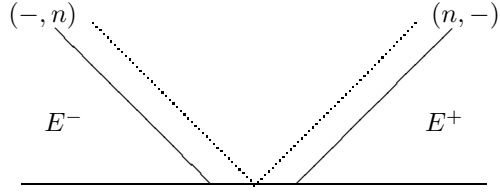
\[
\vcenter{
  \xymatrix @-3.5pc @! {
    &&*{(-,n)}&&&&&&&&&&&&&&&&&&&& *{(n,-)} \\
    &&&&&&&&&&&*{}&&&&&&&&&&&&&&*{}\\
    &&&&&&&&&&*{}&&&&&&&&&&*{}&&&&\\
    &&&&&&&&&&&&&&*{}&&*{}& *{} & *{} & &&&&*{}& \\
    &&&&&&&&&& &  & & &&& & & *{} &&&&*{}&*{}&*{}\\
    &&& *{E^-} &&&&&&*{}&& & & & & *{}& &  & & && *{E^+}&&&\\
    &&&&&&&&&&*{}& & & & *{}  & & & & & &*{}&&&&*{}\\ 
    &&&&&& & & &&&&&& & & &&&& \\
    *{}\ar@{--}[r]&*{}\ar@{-}[rrrrrr]&&&&&&*{} \ar@{-}[rrrrr] &&&\ar@{-}[uuuuuuuullllllll]& & *{} \ar@{.}[uuuuuuuullllllll]\ar@{.}[uuuuuuuurrrrrrrr]\ar@{-}[rr]& *{} & *{} \ar@{-}[uuuuuuuurrrrrrrr]\ar@{-}[rrrr]& & & & *{} \ar@{-}[rrrrr]&*{}&&*{}&&*{}\ar@{--}[r]&*{}\\
           }
}
\]

\end{enumerate}
\end{Theorem}

\begin{proof}
In (i) we need to see that $\cE$ contains all indecomposable objects
of $\cD$, and in (ii) that it contains precisely the indecomposable
objects which are on or below one of the halflines $(-,n)$ and
$(n,-)$.

Each $\ft \in \fT$ divides $\fT$ into two parts: The arcs spanned by
$\ft$ and the rest.  The arcs in $\fT$ which are spanned by $\ft =
(s,t)$ can be viewed as diagonals of a finite polygon $p_\ft$ with
vertices $\{ s, \ldots, t \}$, and as such they form a maximal
collection of non-crossing diagonals of $p_\ft$, that is, a
triangulation of $p_\ft$.

(i)  Let $d \in \cD$ be indecomposable and let $\fd$ be the
corresponding arc.  It is not difficult to show that since $\fT$ is
locally finite, $\fd$ is spanned by an arc $\ft \in \fT$, so $\fd$ can
be viewed as a diagonal of $p_\ft$.  The arcs in $\fT$ which are
spanned by $\ft$ form a triangulation of $p_\ft$, and it is clear
that by finitely many mutations, this triangulation can be changed
into a triangulation containing $\fd$.

But then an equivalent sequence of mutations changes $\fT$ into a
maximal collection of non-crossing arcs which contains $\fd$, and
accordingly, it changes $\cts$ into a cluster tilting subcategory
containing $d$.  So $d \in \cE$.

(ii)  Let $d \in \cD$ be indecomposable and let $\fd = (p,q)$ be the
corresponding arc.  If $d$ is on or below one of the halflines $(-,n)$
and $(n,-)$, then either $p < q \leq n$ or $n \leq p < q$; assume the
former for the sake of argument.  Since $\fT$ has a fountain at $n$,
it follows that there is an $\ft = (m,n)$ in $\fT$ which spans $\fd$.
Now proceed as in part (i) to see $d \in \cE$.

If $d$ is above the halflines $(-,n)$ and $(n,-)$, then $p < n < q$.
This means that $\fd$ crosses an infinite number of arcs in $\fT$, so
there is no finite sequence of mutations which changes $\fT$ into a
maximal collection of non-crossing arcs which contains $\fd$, and
accordingly, no finite sequence of mutations which changes $\cts$ into
a cluster tilting subcategory containing $d$.  Hence $d \not\in \cE$.
\end{proof}

\begin{Proposition}
\label{pro:decoupling}
Suppose that $\fT$ has a fountain at $n$.  Consider the sets $E^-$ and
$E^+$ of indecomposable objects of $\cD$ defined by the figure in
Theorem \ref{thm:rho_domain}(ii).

If $d \in E^-$ then $\CC(d)$ can be written using variables $x_t$
with $t \in \indt \cap E^-$, and if $d \in E^+$ then $\CC(d)$ can
be written using variables $x_t$ with $t \in \indt \cap E^+$.
\end{Proposition}

\begin{proof}
For the sake of argument, suppose $d \in E^-$ and let $\fd$ be the
corresponding arc.

As in the proof of Theorem \ref{thm:rho_domain}(ii), there is a $\ft =
(m,n)$ in $\fT$ which spans $\fd$.  As in the proof of Theorem
\ref{thm:rho_domain}(i), the arcs in $\fT$ which are spanned by $\ft$
can be viewed as a triangulation of the polygon $p_\ft$ with vertices
$\{ m, \ldots, n \}$, and by finitely many mutations, this
triangulation can be changed into a triangulation containing $\fd$.
The arcs $\fc_1, \ldots, \fc_k = \fd$ involved in the mutations are
spanned by $\ft$.

An equivalent sequence of mutations changes $\fT$ into a maximal
collection of non-crossing arcs which contains $\fd$, and accordingly,
it changes $\cts$ into a cluster tilting subcategory containing $d$.
Since the $\fc_i$ are spanned by $\ft$, the corresponding
indecomposable objects $c_1, \ldots, c_k = d$ involved in the
mutations are in $E^-$.

However, $\add E^-$ is closed under extensions since it is equal to
$R^{\perp}$ where $R$ consists of the indecomposable objects on the
half lines $(n+1,-)$ and $(n+2,-)$.  Hence the middle terms of the
exchange triangles $c_i \rightarrow b_i \rightarrow c_{i+1}$ and
$c_{i+1} \rightarrow b_i' \rightarrow c_i$ are direct sums of
indecomposable objects from $E^-$.  Applying Property
\ref{bfhpg:cluster_maps}(iii) of cluster maps successively to these
triangles expresses $\CC(d)$ in terms of variables $x_t$ with $t \in
E^-$.
\end{proof}

\begin{Definition}
For $\obj \in \indt$, let $\ft \in \fT$ be the corresponding arc under
the bijection of Paragraph \ref{bfhpg:A_infty} and set
\[
  \fU(\ft) = \{\, \fu \in \fT \,|\,
                  \mbox{the arc $\fu$ is not spanned by the arc $\ft$} \,\}
           \subseteq \fT.
\]
The bijection described in Paragraph \ref{bfhpg:A_infty} says that
$\fU(\ft)$ corresponds to a set of indecomposable objects $\indu(\obj)
\subseteq \indt$.  Setting $\ctsu(\obj) = \add \indu(\obj)$ gives a
subcategory $\ctsu(\obj) \subseteq \cts$.  
\end{Definition}

Note that $\ft \in \fU(\ft)$ so $\obj \in \ctsu(\obj)$.  Moreover,
there are only finitely many arcs spanned by $\ft$, so $\fU(\ft)$
consists of all but finitely many of the arcs in $\fT$, and hence
$\indu(\obj)$ consists of all but finitely many of the indecomposable
objects in $\indt$.

\begin{Lemma}
\label{lem:Ng}
The subcategory $\ctsu(\obj)$ is functorially finite in $\cD$, and the
Calabi-Yau reduction $\cD_{\indu(\obj)}$ is a cluster category of
Dynkin type $A_m$, where $m$ is the number of arcs in $\fT$ spanned by
$\ft$.
\end{Lemma}

\begin{proof}
Since $\indu(\obj)$ consists of all but finitely many of the
indecomposable objects in $\indt$, it follows from Lemma \ref{lem:
functorially finite} that $\ctsu(\obj)$ is functorially finite in
$\cD$.

The Calabi-Yau reduction $\cD_{\indu(\obj)} = {}^{\perp}(\Sigma
\ctsu(\obj)) / [\ctsu(\obj)]$ is a triangulated $\Hom$-finite
$2$-Calabi-Yau category. Since a Calabi--Yau reduction
of a $2$-Calabi--Yau algebraic category is again algebraic
\cite[Prop 4.3]{AmiotOppermann}, the category $\cD_{\indu(\obj)}$
is algebraic. To see that it is a cluster category of type
$A_m$, by \cite[secs.\ 2 and 4]{KellerReiten} it is sufficient to show
that it has a cluster tilting object $V$ with $\End(V) \cong kA_m$.

By the theory of Calabi-Yau reductions, as explained in Paragraph
\ref{bfhpg:IY}, the cluster tilting subcategories of
$\cD_{\indu(\obj)}$ have the form $\cV/\ctsu(\obj)$ where $\cV$ is a
cluster tilting subcategory of $\cD$ with $\ctsu(\obj) \subseteq \cV$.
Such a $\cV$ corresponds to a maximal collection $\fV$ of non-crossing
arcs with $\fU(\ft) \subseteq \fV$.  Let us construct a set $\fV$ by
starting with $\fU(\ft)$ and adding the following arcs spanned by
$\ft$.
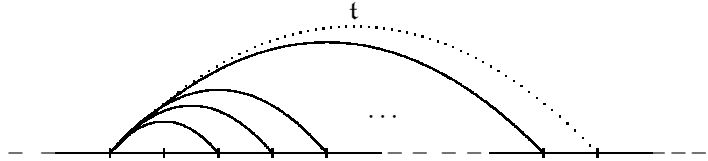
\[
  \def\labelstyle{\textstyle}
  \xymatrix @-4.2pc @! {
       \rule{0ex}{9.5ex} \ar@{--}[r]
     & *{}\ar@{-}[r]
     & *{\rule{0.1ex}{0.8ex}} \ar@{-}[r] \ar@/^1.0pc/@{-}[rr] \ar@/^1.5pc/@{-}[rrr] \ar@/^2.0pc/@{-}^>>>>{\hspace{5ex}\cdots}[rrrr]\ar@/^3.5pc/@{-}[rrrrrrrr]\ar@/^4.0pc/@{.}^{\ft}[rrrrrrrrr]
     & *{\rule{0.1ex}{0.8ex}} \ar@{-}[r] 
     & *{\rule{0.1ex}{0.8ex}} \ar@{-}[r]
     & *{\rule{0.1ex}{0.8ex}} \ar@{-}[r] 
     & *{\rule{0.1ex}{0.8ex}} \ar@{-}[r]
     & {} \ar@{--}[r] 
     & {} \ar@{--}[r]
     & {} \ar@{-}[r]
     & *{\rule{0.1ex}{0.8ex}} \ar@{-}[r] 
     & *{\rule{0.1ex}{0.8ex}} \ar@{-}[r]
     & *{}\ar@{--}[r]
     & *{}
                    }
\]
Then the indecomposable objects of $\cV$ are those of $\ctsu(\obj)$
along with finitely many indecomposable objects placed as follows on a
line segment below $\obj$ in the AR quiver of $\cD$.
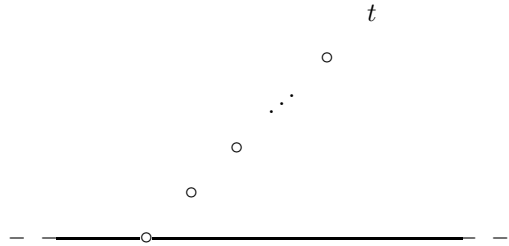
\[
  \xymatrix @-1.55pc @! {
    &&&&&&&& *{\obj} &&&& \\
    &&&&&&& *{\circ} &&&&& \\
    &&&&&& *{\adots} &&&&&& \\
    &&&&&*{\circ}&&&&&&& \\
    &&&&*{\circ}&&&&&&&& \\
    *{}\ar@{--}[r]&*{}\ar@{-}[rr]&&*{\circ}\ar@{-}[rrrrrrr]&&&&&*{}&&*{}\ar@{--}[r]&*{} \\
           }
\]
The cluster tilting subcategory $\cV/\ctsu(\obj)$ of
$\cD_{\indu(\obj)}$ has only finitely many indecomposable objects
given by the indecomposable objects marked by $\circ$ in the figure.
Their direct sum gives a cluster tilting object $V$ of
$\cD_{\indu(\obj)}$, and if there are $m$ of them, then it is not hard
to show that $\End(V) \cong kA_m$.  Here $m$ is the number of arcs in
a triangulation of the polygon $p_{\ft}$, that is, the number of arcs
in $\fT$ spanned by $\ft$.
\end{proof}

\begin{Remark}
The authors believe that all the Calabi--Yau reductions
of $\cD$ are products of cluster categories of type $A$.
If one had a detailed description of rigid subcategories of the
ca\-te\-go\-ry $\cD$, then this could be proven
by using Lemma~\ref{lem: functorially
finite} and the main theorem of~\cite{KellerReiten}. Another possible
approach might be to use the geometric description of the cluster
category of type $A_m$ due to~\cite{CCS}.
\end{Remark}

\begin{Lemma}
\label{lem:t}
Let $d \in \cE$ be an indecomposable object.  There exists $\obj \in
\indt$ such that $d \in {}^{\perp}(\Sigma \ctsu(\obj)) \setminus
\ctsu(\obj)$ and $\CC(d) \in \BQ(x_{\obj'})|_{\obj' \in \indt
\setminus \indu(\obj)}$.
\end{Lemma}

\begin{proof}
Let $\indt'$ be a finite subset of $\indt$ such that $\CC(d)$
can be written in terms of the $x_{\obj'}$ for $\obj' \in \indt'$.  We
must show that there is a $t \in \indt$ such that $d \in
{}^{\perp}(\Sigma \ctsu(\obj)) \setminus \ctsu(\obj)$ and $\indt'
\subseteq \indt \setminus \indu(\obj)$.  Since $\ctsu(t) = \add
\indu(t)$, it is enough to show that there is a $t \in \indt$ such
that $d$ and $\indt'$ are both contained in ${}^{\perp}(\Sigma
\ctsu(\obj)) \setminus \ctsu(\obj)$.

Let us rephrase this in terms of arcs.  Let $d$ correspond to the arc
$\fd$ and $\indt'$ to the finite collection $\fT' \subset \fT$.  We
must show that there is a $\ft \in \fT$ such that that $\fd$ and
$\fT'$ are both contained in the set $\fP(\ft)$ of arcs corresponding
to the indecomposable objects of ${}^{\perp}(\Sigma \ctsu(\obj))
\setminus \ctsu(\obj)$.  Let us show that
\begin{equation}
\label{equ:P}
  \fP(\ft) = \{\, \mbox{$\fp$ is an arc}
               \,|\, \mbox{$\fp$ is spanned by $\ft$} \,\}.
\end{equation}
Namely, \cite[lem.\ 3.6]{HJ} shows that the indecomposable objects in
${}^{\perp}(\Sigma \ctsu(\obj))$ correspond to the arcs which do not
cross any arc in $\fU(\ft)$.  So the indecomposable objects in
${}^{\perp}(\Sigma \ctsu(\obj)) \setminus \ctsu(\obj)$ correspond to
the arcs which do not cross any arc in $\fU(\ft)$ and are outside
$\fU(\ft)$.  Combining this with the definition of $\fU(\ft)$ shows
equality \eqref{equ:P}.

Now, if $\fT$ is locally finite, then it is easy to show that, given
any finite collection $\fQ$ of arcs, there is a $\ft \in \fT$ such
that $\fQ \subseteq \fP(\ft)$.

If $\fT$ has a fountain at $n$, then Theorem \ref{thm:rho_domain}(ii)
implies that the indecomposable object $d$ is in either $E^-$ or
$E^+$.  For the sake of argument, suppose $d \in E^-$.  Then $\indt'$
can be chosen as a subset of $E^-$ by Proposition
\ref{pro:decoupling}.  But then $\fd$ and each $\ft' \in \fT'$ is an
arc of the form $(p,q)$ with $p < q \leq n$.  Since $\fT$ has a
fountain at $n$, it follows that there is an $\ft = (m,n)$ in $\fT$
which spans $\fd$ and each $\ft'$, so $\fd \in \fP(\ft)$ and $\fT'
\subseteq \fP(\ft)$ as desired.
\end{proof}

\begin{Theorem}
\label{thm:rho_properties}
The cluster map $\CC^{\cts} : \objects \cE \rightarrow
\BQ(x_\obj)_{\obj \in \indt}$ enjoys the following properties.
\begin{enumerate}

  \item $\CC^{\cts}(\obj) = x_\obj$ for $\obj \in \indt$.

\smallskip

  \item If $d \in \cE$, then $\CC^{\cts}(d)$ is a non-zero Laurent
    polynomial. 

\smallskip

  \item In each such Laurent polynomial, the coefficients in the
  numerator are positive integers.

\end{enumerate}
\end{Theorem}

\begin{proof}
Part (i) is already in Theorem \ref{thm:cluster_map}.  In (ii) it is
clear from the definition in Paragraph \ref{bfhpg:CC} that
$\CC^{\cts}(d)$ is a Laurent polynomial.

For the rest of the proof, we can assume that $d \in \cE$ is
indecomposable.  Use Lemma \ref{lem:t} to pick $\obj \in \indt$ such
that $d \in {}^{\perp}(\Sigma \ctsu(\obj)) \setminus \ctsu(\obj)$ and
\begin{equation}
\label{equ:CC_in_subfield}
  \CC^{\cts}(d) \in \BQ(x_{\obj'})|_{\obj' \in \indt \setminus \indu(\obj)}.
\end{equation}
The Calabi-Yau reduction $\cD_{\indu(\obj)}$ is a cluster category of
type $A_m$ by Lemma \ref{lem:Ng}, so viewed in $\cD_{\indu(\obj)}$ the
object $d$ is reachable from the cluster tilting subcategory $\cts /
\ctsu(t)$.  Hence Theorem \ref{thm:CC_reduction} gives
$\CC^{\cts}(d)|_{x_{\obju}=1 \:\mbox{\tiny for}\: \obju \in
  \indu(\obj)} = \CC^{\cts/\ctsu(\obj)}(d)$, and by equation
\eqref{equ:CC_in_subfield} this implies that $\CC^{\cts}(d) =
\CC^{\cts/\ctsu(\obj)}(d)$.  But it follows from \cite[thm.\ 4]{Palu}
or Theorem \ref{thm:cluster_map} that $\CC^{\cts/\ctsu(\obj)}$ is a
cluster map on $\cD_{\indu(t)}$, and the remaining part of (ii) and
(iii) now follow from \cite[thm.\ 3]{CK} or
\cite[cor. 3.6]{Schiffler Thomas}.
\end{proof}

\begin{Remark}
By Theorem \ref{thm:rho_domain}(ii), when $\cts$ corresponds to a
fountain, $\cE$ is not all of $\cD$.  It is reasonable to ask in
this case whether $\CC$ can be extended to all of $\cD$.  Let us show
that the answer is {\em no} in the case of the simplest fountain given
by figure \eqref{equ:fountain} of Paragraph \ref{bfhpg:A_infty}.
\end{Remark}

\begin{Theorem}
\label{thm:nogo}
Suppose that $\cts$ corresponds to the fountain \eqref{equ:fountain}.
Then there is no cluster map $\CC : \objects \cD \rightarrow
\BQ(x_\obj)_{\obj \in \indt}$ which satisfies the conditions of
Theorem \ref{thm:rho_properties}.
\end{Theorem}

\begin{proof}
Suppose that such a $\CC$ exists.  By setting all the $x_\obj$ equal
to $1$ we obtain a map $\psi : \objects \cD \rightarrow \{
1,2,3,\ldots \}$.

Moreover, $\indt$ consists of the indecomposable objects on the halflines
$(-,n)$ and $(n,-)$ sketched in Theorem \ref{thm:rho_domain}.  Let us
consider the indecomposables $t_i$ on $(-,n)$ and $u_i$ on the
halfline above it.
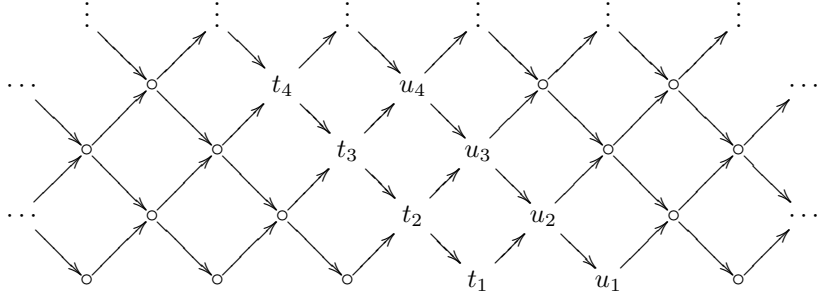
\[
  \xymatrix @-1.7pc @! {
    & \vdots \ar[dr] & & \vdots \ar[dr] & & \vdots \ar[dr] & & \vdots \ar[dr] & & \vdots \ar[dr] & & \vdots & \\
    \cdots \ar[dr]& & *{\circ} \ar[ur] \ar[dr] & & \obj_4 \ar[ur] \ar[dr] & & u_4 \ar[ur] \ar[dr] & & *{\circ} \ar[ur] \ar[dr] & & *{\circ} \ar[ur] \ar[dr] & & \cdots \\
    & *{\circ} \ar[ur] \ar[dr] & & *{\circ} \ar[ur] \ar[dr] & & \obj_3 \ar[ur] \ar[dr] & & u_3 \ar[ur] \ar[dr] & & *{\circ} \ar[ur] \ar[dr] & & *{\circ} \ar[ur] \ar[dr] & \\
    \cdots \ar[ur]\ar[dr]& & *{\circ} \ar[ur] \ar[dr] & & *{\circ} \ar[ur] \ar[dr] & & \obj_2 \ar[ur] \ar[dr] & & u_2 \ar[ur] \ar[dr] & & *{\circ} \ar[ur] \ar[dr] & & \cdots\\
    & *{\circ} \ar[ur] & & *{\circ} \ar[ur] & & *{\circ} \ar[ur] & & \obj_1 \ar[ur] & & u_1 \ar[ur] & & *{\circ} \ar[ur] & \\
               }
\]
We have $\dim_k \Ext^1(u_i,\obj_i) = 1$ for each $i$.  For $i = 1$
there is an AR triangle $\obj_1 \rightarrow u_2 \rightarrow u_1$
and a non-split distinguished triangle $u_1 \rightarrow 0 \rightarrow
\obj_1$; note that the connecting morphism of this triangle is the
isomorphism $\obj_1 \rightarrow \Sigma u_1$.  The defining properties
of cluster maps now imply that $\CC$ satisfies
\[
  \CC(\obj_1)\CC(u_1) = \CC(u_2) + 1.
\]
Likewise, for each $i \geq 2$ there is an AR triangle $\obj_i
\rightarrow \obj_{i-1} \oplus u_{i+1} \rightarrow u_i$ and a non-split
distinguished triangle $u_i \rightarrow 0 \rightarrow \obj_i$; the
connecting morphism of this triangle is the isomorphism $\obj_i
\rightarrow \Sigma u_i$.  Therefore $\CC$ satisfies
\[
  \mbox{$\CC(\obj_i)\CC(u_i) = \CC(\obj_{i-1})\CC(u_{i+1}) + 1$
        for $i \geq 2$}.
\]

The displayed equations can also be written $x_{\obj_1}\CC(u_1) =
\CC(u_2) + 1$ and $x_{\obj_i}\CC(u_i) = x_{\obj_{i-1}}\CC(u_{i+1}) +
1$ for $i \geq 2$.  Setting all the $x_\obj$ equal to $1$, this
implies that $\psi$ satisfies $\psi(u_i) = \psi(u_{i+1}) + 1$ for each
$i \geq 1$.  But this contradicts that the values of $\psi$ are
positive.
\end{proof}

\section{$\SL_2$-tilings}
\label{sec:tilings}

This section has an application of our results to $\SL_2$-tilings as
introduced in \cite{ARS}.
Similar techniques appeared recently in~\cite{AssemDupont}.

\begin{Definition}
\label{def:SL2}
The following notion was introduced in \cite{ARS}: An $\SL_2$-tiling
of $\BZ^2$ with values in a commutative ring $R$ is a map
$r : \BZ \times \BZ \rightarrow R$ for which each matrix 
\[
 \left(
 \begin{array}{cc}
   r(i,j) & r(i,j+1) \\
   r(i+1,j) & r(i+1,j+1) \\
 \end{array}
 \right)
\]
has determinant $1$.  That is,
\begin{equation}
\label{equ:1}
r(i,j)r(i+1,j+1) - r(i,j+1)r(i+1,j) = 1.
\end{equation}

We consider a modified version of this.  An $\SL_2$-tiling of the half
plane $Q = \{\, (m,n) \,\mid\, m \leq n-2 \,\} \subset \BZ^2$ is a map
$r : Q \rightarrow R$ for which the above matrix has determinant $1$
when it makes sense.  Along the edge of $Q$ the matrix does not make
sense, but here we require that
\[
 \left(
 \begin{array}{cc}
   r(i,i+2) & r(i,i+3) \\
   1 & r(i+1,i+3) \\
 \end{array}
 \right)
\]
has determinant $1$ for each $i$.  That is,
\begin{equation}
\label{equ:9}
  r(i,i+2)r(i+1,i+3) - r(i,i+3) = 1.
\end{equation}
\end{Definition}

\begin{Remark}
\label{rmk:Q}
Note that $Q$ consists precisely of the coordinate pairs in the
coordinate system on the AR quiver of $\cD$ given in Paragraph
\ref{bfhpg:A_infty}.  We can view the elements of $Q$ as arcs $(m,n)$
connecting non-neighbouring integers, or as indecomposable objects of
$\cD$.
\end{Remark}

The following result shows that certain patterns of $1$'s distributed
in $Q$ can be extended to $\SL_2$-tilings.

\begin{Theorem}
\label{thm:34}
Let $\fT$ be a locally finite maximal collection of non-crossing arcs.
Then there is an $\SL_2$-tiling $r : Q \rightarrow \{ 1,2,3,\ldots \}$
which satisfies $r(\ft) = 1$ for each $\ft \in \fT$.
\end{Theorem}

\begin{proof}
By Paragraph \ref{bfhpg:A_infty}, the collection $\fT$ corresponds to
the indecomposable objects $\indt$ of a cluster tilting subcategory
$\cts$ of $\cD$, the cluster category of type $A_{\infty}$.  We
must show that there is an $r$ which satisfies $r(\obj) = 1$
for each $\obj \in \indt$.

Theorems \ref{thm:cluster_map} and \ref{thm:rho_domain} say that
$\cts$ gives rise to a cluster map $\CC : \objects\,\cD \rightarrow
\BQ(x_\obj)_{\obj \in \indt}$.
By manipulations with AR triangles like the ones in the proof of
Theorem \ref{thm:nogo}, ones proves that if $(i,i+2)$ is on the edge
of $Q$, then
\[
  \CC(i,i+2)\CC(i+1,i+3) = \CC(i,i+3) + 1.
\]
Likewise, if $(i,j) \in Q$ has $i \leq j-3$, then
\[
  \CC(i,j)\CC(i+1,j+1) = \CC(i,j+1)\CC(i+1,j) + 1.
\]
Hence the map $\CC$, restricted to elements of $Q$, satisfies
equations \eqref{equ:1} and \eqref{equ:9}.

So the map $r$ obtained by setting each $x_{\obj}$ equal to $1$ also
satisfies equations \eqref{equ:1} and \eqref{equ:9}.  Since $\CC(\obj)
= x_{\obj}$ for $\obj \in \indt$ by Theorem \ref{thm:cluster_map}, it
follows that $r(\obj) = 1$ for $\obj \in \indt$.  And the map $r$
takes values in $\{ 1,2,3,\ldots \}$ since each $\CC(i,j)$ is a
non-zero Laurent polynomial whose numerator is a linear combination of
monomials with positive integer coefficients by Theorem
\ref{thm:rho_properties}(ii and iii).
\end{proof}

\begin{Remark}
The main result on $\SL_2$-tilings in \cite{ARS} is that if $S$ is an
admissible frontier of $1$'s in $\BZ^2$, then $S$ can be extended
uniquely to an $\SL_2$-tiling $p$ of $\BZ^2$.

An admissible frontier consists of the coordinate pairs on a doubly
infinite path which is built out of the steps $(1,0)$ and $(0,1)$,
with the stipulation that the steps must alternate infinitely often in
both directions.

Half of a frontier $S$ falls in the half plane $Q$, and it is not hard
to show that the ``zig-zag'' $\indt = S \cap Q$ consists of the
indecomposable objects of a locally bounded cluster tilting
subcategory $\cts$ of $\cD$.  Let us show a sketch of an example.
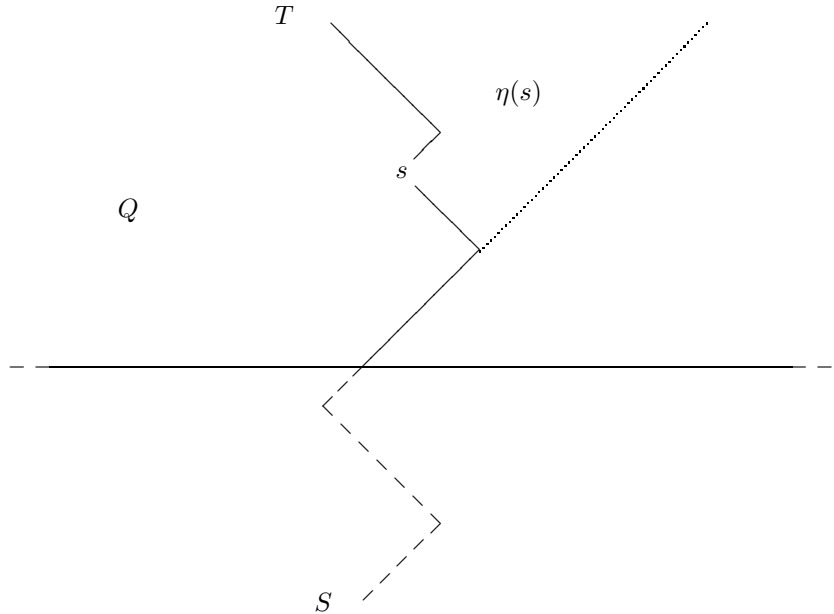
\[
\vcenter{
  \xymatrix @-2.25pc @! {
    &&&&&&& *{T} &&&&&&&&&&&& \\
    &&&&&&&&&&&&&&&&&&& \\
    &&&&&&&&&&&&& *{\eta(s)} &&&&&&&&& \\
    &&&&&&&*{}&&&& *{} \ar@{-}[uuulll] &&&&&&&&&& \\
    &&&&&&&&&& s \ar@{-}[ur] &&&&&&&&&&& \\
    &&& *{Q} &&&&&&&&&&&&&&&&&& \\
    &&&&&&&&&&&& *{} \ar@{-}[uull] \ar@{.}[uuuuuurrrrrr] &&&&&&&&& \\
    &&&&&&&&&&&&&&&&&&&&& \\ 
    &&&&&&&&&&&&&&&&& \\
    *{}\ar@{--}[r]&*{}\ar@{-}[rrrrrr]&&&*{} \ar@{-}[rrrrr] &&&& & *{} \ar@{-}[uuurrr]\ar@{--}[dl]\ar@{-}[rr]&& *{} \ar@{-}[rrrr]& & & & *{} \ar@{-}[rrrrr]&&&&&*{}\ar@{--}[r]&*{}\\
    &&&&&&&& *{} \ar@{--}[dddrrr] &&&&&&\\
    &&&&&&&&&&&&& \\
    &&&&&&&&&&&&& \\
    &&&&&&&&&&& *{} \ar@{--}[ddll] && \\
    &&&&&&&&&&&&& \\
    &&&&&&&& *{S} &&&&& \\
           }
}
\]
By the correspondence between $\indt$'s and $\fT$'s, Theorem
\ref{thm:34} gives an $\SL_2$-tiling $r$ of $Q$ with $r(\obj) = 1$ for
$\obj \in \indt$.

Now, starting from a ``turning point'' $s$ and using that $p$ and
$r$ are $1$ on $\indt$, equation \eqref{equ:1} alone determines $p$
and $r$ in the region $\eta(s)$ indicated in the sketch.  So $p$ and
$r$ agree on $\eta(s)$.  Moreover, the regions $\eta(s)$ cover $\BZ^2$
when $s$ ranges over the turning points of $S$, and for a given $s$,
it is always possible to move $S$ within $\BZ^2$ such that $s$ and
$\eta(s)$ fall inside $Q$.  In this sense, it is possible to obtain
the tiling $p$ by gluing different tilings $r$.

However, the ``zig-zags'' $\indt = S \cap Q$ correspond to very
special examples of locally bounded cluster tilting subcategories.
The corresponding locally finite maximal collections of non-crossing
arcs are built in the following way, with neighbouring arcs going in
one direction from $a$, neighbouring arcs going in the opposite
direction from $b$, etc.
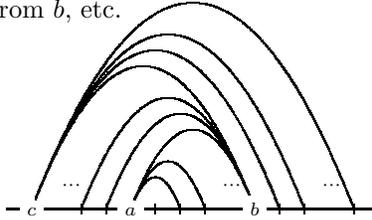
\[
  \def\objectstyle{\scriptstyle}
  \xymatrix @-7.1pc @! {
       \rule{0ex}{15.0ex} \ar@{--}[r]
     & *{}\ar@{-}[r]
     & c \ar@{-}[r] \ar@/^5.0pc/@{-}[rrrrrrrrrr] \ar@/^5.5pc/@{-}^>>>{\hspace{1.5ex}\cdots}[rrrrrrrrrrr] \ar@/^6.5pc/@{-}[rrrrrrrrrrrrr]
     & *{} \ar@{-}[r] 
     & *{\rule{0.1ex}{0.8ex}} \ar@{-}[r] 
     & *{\rule{0.1ex}{0.8ex}} \ar@{-}[r] 
     & a \ar@{-}[r] \ar@/^1.0pc/@{-}[rr] \ar@/^1.5pc/@{-}^>>>{\hspace{1.5ex}\cdots}[rrr] \ar@/^2.5pc/@{-}[rrrrr]
     & *{\rule{0.1ex}{0.8ex}} \ar@{-}[r] 
     & *{\rule{0.1ex}{0.8ex}} \ar@{-}[r]
     & *{\rule{0.1ex}{0.8ex}} \ar@{-}[r] 
     & *{} \ar@{-}[r] 
     & b \ar@{-}[r] \ar@/^-3.0pc/@{-}[llllll] \ar@/^-3.5pc/@{-}_>>>{\hspace{-85ex}\cdots}[lllllll] \ar@/^-4.5pc/@{-}[lllllllll]
     & *{\rule{0.1ex}{0.8ex}} \ar@{-}[r] 
     & *{\rule{0.1ex}{0.8ex}} \ar@{-}[r]
     & *{} \ar@{-}[r] 
     & *{\rule{0.1ex}{0.8ex}} \ar@{-}[r]
     & *{}\ar@{--}[r]
     & *{}
                    }
\]
It is clear, for instance by mutating a configuration like this, that
much more general $\indt$'s are possible, and Theorem
\ref{thm:34} also works for these.

It would be interesting to obtain a version of Theorem \ref{thm:34} with
$\BZ^2$ instead of $Q$.
\end{Remark}

\end{document}